\documentclass[12pt]{amsart}

\usepackage{graphicx} 
\usepackage{scrextend}
\usepackage{amsfonts, amsmath, amssymb,amsthm,comment}  
\usepackage{times,enumerate}
\usepackage{mathdots}%
\usepackage{nccmath}
\usepackage{mathtools}
\usepackage{multicol}
\newenvironment{mpmatrix}{\begin{medsize}\begin{pmatrix}}%
    {\end{pmatrix}\end{medsize}}%
\usepackage{dsfont,bbm}
\usepackage{rotating}
\usepackage{lscape}
\usepackage{url}
\usepackage{multicol}
\usepackage{thmtools, thm-restate}
\usepackage{blkarray}
\usepackage{multicol}
\usepackage[margin=2.5cm]{geometry}
 \usepackage[foot]{amsaddr}
 \usepackage{hyperref}
\usepackage{cancel}

\usepackage[usenames, dvipsnames]{xcolor}
\DeclareGraphicsExtensions{.pdf,.png,.jpg}

\hypersetup{
    colorlinks,
    linkcolor={blue},
    citecolor={blue},
    urlcolor={blue}
}

\DeclareMathOperator{\Span}{span}

\DeclareMathOperator{\Rank}{rank}

\makeatletter
\def\widebreve{\mathpalette\wide@breve}
\def\wide@breve#1#2{\sbox\z@{$#1#2$}%
     \mathop{\vbox{\m@th\ialign{##\crcr
\kern0.08em\brevefill#1{0.8\wd\z@}\crcr\noalign{\nointerlineskip}%
                    $\hss#1#2\hss$\crcr}}}\limits}
\def\brevefill#1#2{$\m@th\sbox\tw@{$#1($}%
  \hss\resizebox{#2}{\wd\tw@}{\rotatebox[origin=c]{90}{\upshape(}}\hss$}
\makeatletter

\newcommand{\RR}{\mathbb R}

\newcommand{\NN}{\mathbb N}
\newcommand{\ZZ}{\mathbb Z}
\newcommand{\CC}{\mathbb C}

\newcommand{\Mod}[1]{\ (\mathrm{mod}\ #1)}

\newcommand{\cB}{\mathcal B}

\newcommand{\cA}{\mathcal A}

\newcommand{\cZ}{\mathcal Z}

\newcommand{\cC}{\mathcal C}

\newcommand{\benu}{\begin{enumerate}}
\newcommand{\eenu}{\end{enumerate}}
\newcommand{\bop}{\begin{opomba}}
\newcommand{\eop}{\end{opomba}}

\newcommand{\supp}{\mathrm{supp}}

\newtheorem{theorem}{Theorem}[section]
\newtheorem{corollary}[theorem]{Corollary}
\newtheorem{lemma}[theorem]{Lemma}
\newtheorem{proposition}[theorem]{Proposition}

\theoremstyle{definition}

\newtheorem{example}[theorem]{Example}

\newcommand{\mbf}{\mathbf}

\theoremstyle{remark}
\newtheorem{remark}[theorem]{Remark}

\numberwithin{equation}{section}

\begin{document}

\title{The truncated Hamburger moment problems with gaps in the index set}

\author[Alja\v z Zalar]{Alja\v z Zalar}

\address{%
Faculty of Computer and Information Science\\
University of Ljubljana\\
Ve\v cna pot 113\\
1000 Ljubljana\\
Slovenia}

\email{aljaz.zalar@fri.uni-lj.si}

\thanks{Supported by the Slovenian Research Agency grants J1-2453, P1-0288.}

\subjclass[2010]{Primary 47A57, 47A20, 44A60; Secondary 
15A04, 47N40.}

\keywords{Hamburger moment problem, truncated moment problems, representing measure, moment matrix}
\date{July 21, 2020}
\maketitle

\begin{abstract}

	In this article we solve four special cases of the truncated Hamburger moment problem (THMP) of degree $2k$ with one or two missing
	moments in the sequence. As corollaries we obtain, by using appropriate substitutions, 
	the solutions to bivariate truncated moment problems of degree $2k$
	for special curves. Namely, for the curves $y=x^3$ (first solved by Fialkow \cite{Fia11}), 
	$y^2=x^3$, $y=x^4$ where a certain moment of degree $2k+1$ is known and $y^3=x^4$ with a certain moment given.  
	The main technique is the completion of the partial positive semidefinite matrix (ppsd) such that the conditions of Curto and Fialkow's solution of 
	the THMP are satisfied. The main tools are the use of the properties of positive semidefinite Hankel matrices and
	a result on all completions of a ppsd matrix with one unknown entry, proved by the use of the Schur complements for
	$2\times 2$ and $3\times 3$ block matrices.

\end{abstract}

\section{Introduction}

For $x=(x_1,\ldots,x_d)\in \RR^d$ and $i=(i_1,\ldots,i_d)\in \ZZ^d_+$, we set $|i|=i_1+\ldots+i_d$ and $x^i=x_1^{i_1}\cdots x_d^{i_d}$.
Given a real $d$-dimensional multisequence $\beta=\beta^{ (2k)}=\{\beta_i\}_{i\in \ZZ_+^d,|i|\leq 2k}$ of degree $2k$
and a closed subset $K$ of $\RR^d$, the \textbf{truncated moment problem (TMP)} supported on $K$ for $\beta$
asks to characterize the existence of a positive Borel measure $\mu$ on $\RR$ with support in $K$, such that
	\begin{equation}\label{moment-measure-cond}
		\beta_i=\int_{K}x^id\mu(x)\quad \text{for}\quad i\in \ZZ^d_+, |i|\leq 2k.
	\end{equation}
If such measure exists, we say that $\beta$ has a representing measure supported on $K$ and $\mu$ is its $K$-\textbf{representing measure}.

We denote by $M(k)=M(k)(\beta)=(\beta_{i,j})_{i,j=0}^{k}$ the moment matrix associated with $\beta$, where the rows and columns are indexed by $X^i$, $|i|\leq k$, in degree-lexicographic order.
Let $\RR[x]_k:=\{p\in \RR[x]\colon \deg p\leq k\}$ stand for the set of polynomials in $d$ variables
of degree at most $k$.
To every $p:=\sum_{i\in \ZZ^d_+, |i|\leq k} a_ix^i\in \RR[x]_k$, we denote by $p(X)=\sum_{i\in \ZZ^d_+, |i|\leq k} a_iX^i$ the vector from the column
space $\cC(M(k))$ of the matrix $M(k)$. 
Recall from \cite{CF96}, that $\beta$ has a representing measure $\mu$ with the support $\supp \;\mu$ being a 
subset of 
$\cZ_p:=\{x\in \RR^d\colon p(x)=0\}$ if and only if $p(X)=0$.
We say that the matrix $M(k)$ is \textbf{recursively generated (rg)} if for $p,q,pq\in \RR[x]_k$ such that $p(X)=0$, 
it follows that $(pq)(X)=0.$

%
%
The full moment problem (MP), where $\beta_i$ is given for every $i\in \ZZ^d_+$, being the classical question in analysis and also due to its relation
with real algebraic geometry via the duality with positive polynomials given by Haviland's theorem \cite{Hav35}, 
has been widely studied, see e.g., 
	\cite{Akh65,AhK62,KN77,Las09,Lau05,Lau09,Mar08,PS06,PS08,Put93,PV99,Sch91,Sch03,Sch17}.
The TMP, which is more general than the full MP \cite{Sto01}, has been intensively studied in a series of papers by Curto and Fialkow \cite{CF91, CF96, CF98a, CF98b,CF02,CF04,CF05,CF08} with the celeberated flat extension theorem they established as a core tool in the field.
There are also various generalizations of the TMP (e.g., \cite{Dym89,Bol96,BW06,DU18}, to matrix moments, \cite{BK10,BK112} to tracial moments, \cite{IKLS17} to infinitely many variables). Recently, Fialkow's core variety \cite{Fia17} approach led to many new results on the TMP;
see also \cite{BF20,DS18}. 
A \textbf{concrete solution} to the TMP is a set of necessary and sufficient conditions for the existence of a $K$-representing measure. Among necessary conditions, $M(k)$ must be psd and rg \cite{CF91,CF98b}, which also suffice in some cases. Concrete solutions to the TMP are known in the following cases:
\begin{enumerate}
  \item\label{THMP-pt1} (Truncated Hamburger moment problem (THMP)) $d=1$ and $K=\RR$. 
	See \cite[Theorem I.3]{AhK62} 	or \cite[Theorem A.II.1]{Ioh82} for the special case of even $k$
	with an invertible moment matrix and \cite[Section 3]{CF91} for the general case.
  \item\label{pt2-13-07} (Truncated Hausdorff moment problem) $d=1$ and $K=[0,\infty)$.
	See \cite[p.\ 175]{KN77}  for the special case of an invertible moment matrix and \cite[Section 5]{CF91} for the general case.
  \item\label{pt3-13-07} (Truncated Stieltjes moment problem) $d=1$ and $K=[a,b]$, $a<b$. 
	See \cite[Theorems III.2.4 and II.2.3]{KN77} and \cite[Section 4]{CF91}  for the general case. 
  \item\label{quartic-13-07} $d=2$ and $K$ is a curve $p(x,y)=0$ with $\deg p\leq 2$. See \cite{CF02, CF04, CF05,FN10,Fia14,CS16}.
  \item\label{case-y-x3} $d=2$ and $K$ is a curve $y=x^3$. See \cite{Fia11}.
  \item\label{rec-deter-13-07} $d=2$ and the moment matrix has a special feature called \textit{recursive determinateness}. 
	See \cite{CF13} for details.
  \item (Extremal case) The rank of the moment matrix is the same as the cardinality of the corresponding variety;
	see \cite{CFM08}.
  \item\label{THMP-pt2} Some special cases are solved in \cite{CS15,Fia17,Ble15,BF20}.
\end{enumerate}

In \eqref{case-y-x3}, $\beta$ must satisfy certain numerical conditions, which are equivalent to the conditions from
Corollary \ref{Y=X3-gen-bg} below. The proof is by separating the nonsingular case from the singular one. 
In the nonsingular case the existence of a flat extension is established by a detailed and technically demanding analysis,
while the singular case is done by the use of additional features of the moment matrix such as recursive determinateness and known results for such matrices.

In this article we present concrete solutions to the four cases of the THMP of degree $2k$ with some unknown moments $\beta_{i_1},\ldots,\beta_{i_j}$, $1\leq i_1\leq\cdots\leq i_j\leq 2k-1$, in the sequence, which we 
call the \textbf{THMP with gaps ($\beta_{i_1},\ldots,\beta_{i_j}$)}.
Namely, we solve the THMP with gaps $(\beta_{2k-1})$, $(\beta_{2k-2},\beta_{2k-1})$, $(\beta_1)$ and 
$(\beta_1,\beta_2)$. The motivation to solve this cases of the THMP with gaps is to obtain the solutions to the special cases of the 2-dimensional TMP.
Namely, the solution of the THMP with gaps:
\begin{enumerate} 
  \item $(\beta_{2k-1})$ gives an alternative solution to the TMP with $d=2$ and $K$ being the curve $y=x^3$ (see \eqref{case-y-x3} above). 
	The advantage of our approach is that the proof is short and we also do not need to separate three subcases, i.e., $k=1$, $k=2$ and $k\geq 3$. 
  \item $(\beta_{2k-2},\beta_{2k-1})$ solves the TMP with $d=2$, $K$ being the curve $y=x^4$ and in addition the 	
	moment $\beta_{3,2k-2}$ of degree $2k+1$ is known. To solve the TMP for the curve $y=x^4$ without this
	additional moment, one needs to solve the THMP with gaps $(\beta_{2k-5}, \beta_{2k-2},\beta_{2k-1})$ 
	which is a possible topic of future research. 
  \item $(\beta_1)$ solves the TMP with $d=2$ and $K$ being the curve $y^2=x^3$. 
  \item $(\beta_1,\beta_2)$ solves the TMP with $d=2$, $K$ being the curve $y^3=x^4$ and known 
	$\beta_{\frac{5}{3},0}$. By $\beta_{\frac{5}{3},0}$ we mean the moment of  
	$x_1^{\frac{5}{3}}$, i.e., $\int_K x_1^{\frac{5}{3}}d\mu$. To solve the TMP for the curve $y^3=x^4$ without
	this additional information, one needs to solve the THMP with gaps $(\beta_{1}, \beta_{2},\beta_{5})$, which 
	is another open question for future research.
\end{enumerate}

\subsection{ Readers Guide}
The paper is organized as follows. 
In Section \ref{S2} we present the tools used in the proofs of our main results:
\begin{itemize}
	\item Generalized Schur complements and verification of positive semidefiniteness of block matrices 
		(Subsection \ref{SubS2.1}). 
	\item Properties of psd Hankel matrices (Subsection \ref{SubS2.2}).
	\item The solution to the THMP (Subsection \ref{SubS2.3}).
	\item A result about psd completions of partial psd matrices 
			with one unknown entry (Subsetion \ref{SubS2.4}).
	\item An extension principle for psd matrices (Subsection \ref{SubS2.5}).
	\item A result about subsequences of moment sequences (Subsection \ref{SubS2.6}).
\end{itemize}
In Section \ref{S3} we solve the THMP of degree $2k$ with gaps $(\beta_{2k-1})$ 
	(see Theorem \ref{trunc-Hamb-without-3n-1-bg})
and $(\beta_{2k-2},\beta_{2k-1})$ 
	(see Theorem \ref{trunc-Hamb-without-3n-2-and-1}). 
Corollary \ref{Y=X3-gen-bg}, being a special case of the $(\beta_{2k-1})$-case, 
is the solution to the TMP with $d=2$ and the curve $y=x^3$ as $K$,
while Corollary \ref{Y=X4-gen}, being a special case of the $(\beta_{2k-2},\beta_{2k-1})$-case, 
is the solution to the TMP with $d=2$, the curve $y=x^4$ as $K$ and an additional moment $\beta_{3,2k-2}$ known.

In Section \ref{S4} we solve the THMP of degree $2k$ with gaps $(\beta_{1})$ 
	(see Theorem \ref{trunc-Hamb-without-1})
and $(\beta_{1},\beta_{2})$ 
	(see Theorem \ref{trunc-Hamb-without-1-2}). 
Corollary \ref{Y2=X3-general}, being a special case of the $(\beta_{1})$-case, 
is the solution to the TMP with $d=2$ and the curve $y^2=x^3$ as $K$,
while Corollary \ref{Y3=X4-gen}, being a special case of the $(\beta_{1},\beta_{2})$-case, 
is the solution to the TMP with $d=2$, the curve $y^3=x^4$ as $K$ and an additional moment $\beta_{\frac{5}{3},0}$ known.\\

\noindent \textbf{Acknowledgement}. I would like to thank Jaka Cimpri\v c and Abhishek Bhardwaj for 
useful suggestions on the preliminary versions of this article.

\section{Preliminaries}\label{S2}

In this section we present some tools which will be needed in the proofs of our main results in
Sections \ref{S3} and \ref{S4}.\\

We write $M_{n,m}$ (resp.\ $M_n$) for the set of $n\times m$ (resp.\ $n\times n$) real matrices. 
For a matrix $M$ we denote by $\cC(M)$ its column space.
The set of real symmetric matrices of size $n$ will be denoted by $S_n$. 
For a matrix $A\in S_n$ the notation $A\succ 0$ (resp.\ $A\succeq 0$) means $A$ is positive definite (pd) (resp.\ positive semidefinite (psd)).

\subsection{Generalized Schur complements}\label{SubS2.1}
Let 
	\begin{equation}\label{matrixM}
		M=\left( \begin{array}{cc} A & B \\ C & D \end{array}\right)\in S_{n+m}
	\end{equation}
be a real matrix where $A\in M_n$, $B\in M_{n,m}$, $C\in M_{m,n}$  and $D\in M_{m}$.
The \textbf{generalized Schur complement} \cite{Zha05} of $A$ (resp.\ $D$) in $M$ is defined by
	$$M/A=D-CA^+B\quad(\text{resp.}\; M/D=A-BD^+C),$$
where $A^+$ (resp.\ $D^+$) stands for the Moore-Penrose inverse of $A$ (resp.\ $D$). 

\begin{remark}\label{rem-10-07-20}
	\begin{enumerate}
	  \item If $A$ (resp.\ $D$) is invertible, then $M/A$ (resp.\ $M/D$) is the usual Schur complement of $A$ (resp.\ $D)$ in $M$.
	  \item \label{SC-remark} Note that $M/A=\left( \begin{array}{cc} D & C \\ B & A\end{array}\right)/A$.
	\end{enumerate}
\end{remark}

The following theorem gives conditions for verifying positive semidefiniteness of a block matrix of size 2.

\begin{theorem}\label{block-psd} \cite{Alb69} 
	Let 
		\begin{equation}\label{form-of-M}
			M=\left( \begin{array}{cc} A & B \\ B^{T} & C\end{array}\right)\in S_{n+m}
		\end{equation} 
	be a real symmetric matrix where $A\in S_n$, $B\in M_{n,m}$ and $C\in S_m$.
	Then the following conditions are equivalent:
	\begin{enumerate}
	  \item $M\succeq 0$ .
	  \item $C\succeq 0$, $\cC(B^T)\subseteq\cC(C)$ and $M/C\succeq 0$.
	  \item $A\succeq 0$, $\cC(B)\subseteq\cC(A)$ and $M/A\succeq 0$.
	\end{enumerate}
\end{theorem}

If $m=1$ in \eqref{form-of-M}, then $\Rank M\in \{\Rank A,\Rank A+1\}$.
The following proposition characterizes w.r.t.\ the value of $M/A$ when each of the possibilities occurs in the case $M$ is psd.

\begin{proposition}\label{rank-13-07}
	Let 
		\begin{equation*}
			M=\left( \begin{array}{cc} A & b \\ b^{T} & c\end{array}\right)\in S_{n+1}
		\end{equation*} 
	be a real symmetric matrix where $A\in S_n$, $b\in \RR^n$ and $c\in \RR$.
	Then $\Rank M=\Rank A$ if and only if $M/A=0$. Otherwise $\Rank M=\Rank A+1$.
\end{proposition}

\begin{proof}
	By Theorem \ref{block-psd}, the psd assumption implies that $b\in \cC(A)$.
	By the properties of the Moore-Penrose inverse 	$\{A^+b+w\colon w\in \ker A\}$ is the set of solutions $z$ of the system $Az=b$.
	Therefore, 
	\begin{equation}\label{col-space-M}
		\cC(M)=
		\cC\big(\left(\begin{array}{cc}
			A & 0 \\
			b^T & c-b^T(A^+b+w)
		\end{array}\right)\big)=
	\cC\big(\left(\begin{array}{cc}
			A & 0 \\
			b^T & M/A
		\end{array}\right)\big),
	\end{equation}
	where the second equality follows from the fact that $A$ is symmetric, 
	$b\in \cC(A)$ and $w\in \ker A$.
	Now, the statement of the proposition follows from \eqref{col-space-M}.
\end{proof}

The following proposition gives an explicit formula, called the \textit{quotient formula} \cite{CH69},  
for expressing the Schur complement of a $2\times 2$ upper left-hand or a $2\times 2$ lower right-hand block in 
a $3\times 3$ block matrix using $2\times 2$ block submatrices.

\begin{proposition}\label{Schur3by3-ver2}
	Let 
	\begin{equation*}\label{matrixK}
		K
		=\left( \begin{array}{ccc} A & B & D \\ B^T & C & E \\ D^T & E^T & F\end{array}\right)
		=\left( \begin{array}{c|c} M & \begin{array}{c} D\\ E \end{array} \\ \hline 
			\begin{array}{cc} D^T & E^T\end{array} & F \end{array}\right)
		=\left( \begin{array}{c|c} A & \begin{array}{cc} B & D \end{array} \\ \hline 
			\begin{array}{c} B^T \\ D^T\end{array} & N \end{array}\right) \in S_{n_1+n_2+n_3}
	\end{equation*}
	be a $3\times 3$ block real matrix, where $A\in S_{n_1}, C\in S_{n_2}, F\in S_{n_3}$ are real symmetric matrices and 
	$B\in M_{n_1,n_2},D_{n_1,n_3},E_{n_2,n_3}$ are rectangular matrices.
	If $M$ and $A$ are nonsingular, then
		\begin{equation}\label{qf1}
			K/M=\left(\begin{array}{cc} A & D \\ D^T & F \end{array}\right)/A
					- \left[\left(\begin{array}{cc} A & B \\ D^T & E^T \end{array}\right)\Big/A\right]
						(M/A)^{-1}
						\left[\left(\begin{array}{cc} A & D \\ B^T & E \end{array}\right)\Big/A\right].
		\end{equation}
	If $N$ and $C$ are nonsingular, then
		\begin{equation}\label{qf2}
			K/N=
				 \left(\begin{array}{cc} C & B^T \\ B & A \end{array}\right)\Big/C 
					- \left[\left(\begin{array}{cc} C & E \\ B & D \end{array}\right)\Big/C\right]
						(N/C)^{-1}
						\left[\left(\begin{array}{cc} C & B^T \\ E^T & D^T \end{array}\right)\Big/C\right].
		\end{equation}
\end{proposition}

\begin{proof}
	By an easy calculation we have that
		$$K/A=\left( \begin{array}{cc} M/A & \left(\begin{array}{cc} A & D \\ B^T & E \end{array}\right)/A\\
				 \left(\begin{array}{cc} A & B \\ D^T & E^T \end{array}\right)/A &
				 \left(\begin{array}{cc} A & D \\ D^T & F \end{array}\right)/A 
				\end{array}\right).$$
	Now the quotient formula \cite{CH69} $K/M=(K/A)/(M/A)$ yields \eqref{qf1}. 

	By Remark \ref{rem-10-07-20} \eqref{SC-remark}, it is true that $K/N=L/N$ where 
		$$L=
		\left( \begin{array}{c|c} N & \begin{array}{c} B^T \\ D^T \end{array} \\ \hline 
			\begin{array}{cc} B & D\end{array} & A \end{array}\right).$$
	Now \eqref{qf2} follows from \eqref{qf1}.
\end{proof}					

\subsection{Hankel matrices}\label{SubS2.2}
Let $k\in \NN$.
For 
	\begin{equation*}\label{vector-v}
		\beta=(\beta_0,\ldots,\beta_{2k} )\in \RR^{2k+1},
	\end{equation*}
we denote by
	\begin{equation*}\label{vector-v}
		A_{\beta}:=\left(\beta_{i+j} \right)_{i,j=0}^k
					=\left(\begin{array}{ccccc} 
							\beta_0 & \beta_1 &\beta_2 & \cdots &\beta_k\\
							\beta_1 & \beta_2 & \iddots & \iddots & \beta_{k+1}\\
							\beta_2 & \iddots & \iddots & \iddots & \vdots\\
							\vdots 	& \iddots & \iddots & \iddots & \beta_{2k-1}\\
							\beta_k & \beta_{k+1} & \cdots & \beta_{2k-1} & \beta_{2k}
						\end{array}\right)
					\in S_{k+1}
	\end{equation*}
the corresponding Hankel matrix. We denote by 
	$\mbf{v_j}:=\left( \beta_{j+\ell} \right)_{\ell=0}^k$ the $(j+1)$-th column of $A_{\beta}$, $0\leq j\leq k$, i.e.,
		$$A_{\beta}=\left(\begin{array}{ccc} 
								\mbf{v_0} & \cdots & \mbf{v_k}
							\end{array}\right).$$
As in \cite{CF91}, the \textbf{rank} of $\beta$, denoted by $\Rank \beta$, is defined by
	$$\Rank \beta=
	\left\{\begin{array}{rl} 
		k+1,&	\text{if } A_{\beta} \text{ is nonsingular},\\
		\min\left\{i\colon \bf{v_i}\in \Span\{\bf{v_0},\ldots,\bf{v_{i-1}}\}\right\},&		\text{if } A_{\beta} \text{ is singular}.
	 \end{array}\right.$$
We denote the upper left-hand corner of $A_{\beta}$ of size $m+1$ by 
		$$A_{\beta}(m)=\left(\beta_{i+j} \right)_{i,j=0}^m\in S_{m+1}.$$
The following proposition is the alternative description of $\Rank \beta$ if $A_\beta$ is singular.

\begin{proposition}\label{alternative-rank}\cite[Proposition 2.2]{CF91} 
	 Let $k\in \NN$, 
	$\beta=(\beta_0,\ldots,\beta_{2k})$,
	and assume that $A$ is positive semidefinite and singular.
	Then
		$$\Rank \beta=\min\{j\colon 0\leq j\leq k \text{ such that }A_\beta(j) 		
			\text{ is singular}\}.$$
\end{proposition}

%

Important property of psd Hankel matrices is the following rank principle.

\begin{theorem}\label{rank-principle} \cite[Corollary 2.5]{CF91} 
	 Let $k\in \NN$, 
	$\beta=(\beta_0,\ldots,\beta_{2k})$, 
	$\widetilde \beta=(\beta_0,\ldots, \beta_{2k-2})$,
	$A_\beta\succeq 0$ and $r=\Rank \widetilde \beta$.
	Then:
  \begin{enumerate}
	\item\label{pt1-10-07-20} $\Rank A_{\widetilde \beta}=r$.
	\item $r\leq \Rank A_\beta \leq r+1$.
	\item $\Rank A_\beta = r+1$ if and only if 
				$$\beta_{2k}> \varphi_0 \beta_{2k-r}+\ldots+
					\varphi_{r-1}\beta_{2k-1},$$
	  where
			$(\varphi_0,\ldots,\varphi_{r-1}):=
			A_{\beta}(r-1)^{-1}(\beta_r,\ldots,\beta_{2r-1})^{T}$ 
  \end{enumerate}
\end{theorem}

We will use the following corollary of Proposition \ref{alternative-rank} and Theorem \ref{rank-principle} in the sequel.

\begin{corollary}\label{rank-theorem-2}
	In the notation of Theorem \ref{rank-principle}, under the assumptions 
	$A_\beta\succeq 0$, $A_\beta$ is singular, and $r=\Rank \widetilde\beta$, then
	  $$r=\Rank \beta=\Rank A_\beta(r-1)=\Rank A_\beta(r)=\ldots=
	  	\Rank A_{\beta}(k-1)=\Rank A_{\widetilde{\beta}}.$$
\end{corollary}

We denote the lower right-hand corner of $A_{\beta}$ of size $m+1$ by 
	$$A_{\beta}[m]=\left(\beta_{i+j} \right)_{i,j=m-k}^k
		=\left(\begin{array}{ccccc} 
							\beta_{2(k-m)} & \beta_{2(k-m)+1} &\beta_{2(k-m+1)}& \cdots &\beta_{2k-m}\\
							\beta_{2(k-m)+1} & \beta_{2(k-m+1)} & \iddots & \iddots & \beta_{2k-m+1}\\
							\beta_{2(k-m+1)} & \iddots & \iddots & \iddots & \vdots\\
							\vdots 	& \iddots & \iddots & \iddots & \beta_{2k-1}\\
							\beta_{2k-m} & \beta_{2.k-m+1} & \cdots & \beta_{2k-1} & \beta_{2k}
			\end{array}\right)
		\in S_{m+1}$$
Let 
	$$\beta^{\text{(rev)}}:=(\beta_{2k},\beta_{2k-1},\ldots,\beta_0)$$
be the sequence obtained from $\beta$ by reversing the order of numbers.
Using Corollary \ref{rank-theorem-2} for a reversed sequence implies the following corollary.

\begin{corollary}\label{rank-theorem-3}
	In the notation of Theorem \ref{rank-principle}, under the assumption 
	$A_\beta\succeq 0$, $A_\beta$ is singular and $r=\Rank \widetilde{\beta}^{\text{(rev)}}$, 
	where $\widetilde{\beta}^{\text{(rev)}}:=(\beta_{2k},\ldots,\beta_2)$,
	it holds that 
	$$r=\Rank \beta^{\text{(rev)}}=\Rank A_\beta[r-1]=\Rank A_\beta[r]=\ldots=
		\Rank A_{\beta}[k-1]=\Rank A_{\widetilde{\beta}^{\text{(rev)}}}.$$
\end{corollary}

\begin{proof}
	Corollary \ref{rank-theorem-2} used for ${\beta}^{\text{(rev)}}$ implies that
		\begin{equation}\label{rank-chain}
			r=\Rank {\beta}^{\text{(rev)}}=\Rank A_{\beta^{\text{(rev)}}}(r-1)=\Rank A_{\beta^{\text{(rev)}}}(r)=\ldots=
				A_{\beta^{\text{(rev)}}}(k-1)=\Rank A_{\widetilde\beta^{\text{(rev)}}}.
		\end{equation}
	For $\ell=0,\ldots,k$ define the permutation matrices 
		$P_{\ell}:\RR^{\ell+1}\to \RR^{\ell+1}$ by $e^{(\ell)}_{i}\mapsto e^{(\ell)}_{\ell+2-i}$, $i=1,\ldots,\ell+1$,
	where $e^{(\ell)}_1,\ldots,e^{(\ell)}_{\ell+1}$ is the standard basis for $\RR^{\ell+1}$.
	Note that 
		$A_{\beta^{\text{(rev)}}}(\ell)=P_{\ell}^TA_{\beta}[\ell]P_{\ell}$
	and hence 
		$\Rank A_{\beta^{\text{(rev)}}}(\ell)=\Rank A_{\beta}[\ell]$, which together with \eqref{rank-chain}
	implies the statement of the corollary.
\end{proof}

A sequence $\beta=(\beta_0,\ldots,\beta_{2k})$ with $r:=\Rank \beta$ is \textbf{positively recursively generated}
if $A_\beta(r-1)\succ 0$ and denoting $(\varphi_0,\ldots,\varphi_{r-1}):=A_{\beta}(r-1)^{-1}(\beta_r,\ldots,\beta_{2r-1})^{T}$, it is true that
\begin{equation}\label{recursive-generation}
  \beta_j=\varphi_0\beta_{j-r}+\cdots+\varphi_{r-1}\beta_{j-1}\quad \text{for}\quad j=r,\ldots,2k.
\end{equation}
Note that \eqref{recursive-generation} is equivalent to
\begin{equation}\label{recursive-generation-equivelant}
  \mbf{v_j}=\varphi_0 \mbf{v_{j-r}}+\cdots+\varphi_{r-1}\mbf{v_{j-1}}\quad \text{for}\quad j=r,\ldots,k.
\end{equation}

\subsection{Solution of the truncated Hamburger moment problem}\label{SubS2.3}

\begin{theorem}\label{Hamburger}\cite[Theorem 3.9]{CF91}
	For $k\in \NN$ and $\beta=(\beta_0,\ldots,\beta_{2k})$ with $\beta_0>0$, the following statements are equivalent:
\begin{enumerate}	
	\item There exists a representing measure for $\beta$ supported on $K=\RR$.
	\item There exists a $(\Rank \beta)$-atomic representing measure for $\beta$.
	\item $\beta$ is positively recursively generated.
	\item\label{pt4-v2206} $A_\beta\succeq 0$ and $\Rank A_\beta=\Rank \beta$.
\end{enumerate}
\end{theorem}

A straightforward corollary of Theorem \ref{Hamburger} and Corollary \ref{rank-theorem-2} is the following.

\begin{corollary}\label{singular-case-measure}
	Let $k\in \NN$ and $\beta=(\beta_0,\ldots,\beta_{2k})$ with $\beta_0>0$. Suppose that $A_{\beta}$ is singular.
	The following statements are equivalent:
	\begin{enumerate}
		\item There exists a representing measure for $\beta$ supported on $K=\RR$.
		\item There exists a $(\Rank \beta)$-atomic representing measure for $\beta$.
		\item $\beta$ is positively recursively generated.
		\item $A_{\beta}\succeq 0$ and $\Rank A_{\beta}=\Rank A_{\beta}(k-1)$.
	 \end{enumerate}
\end{corollary}

\subsection{Partially positive semidefinite matrices and their completions}\label{SubS2.4}

A \textbf{partial matrix} $A=(a_{ij})_{i,j=1}^n$ is a matrix of real numbers $a_{ij}\in \RR$, where some of the entries are not specified. 

A partial symmetric matrix $A=(a_{ij})_{i,j=1}^n$ is 
\textbf{partially positive semidefinite (ppsd)} 
(resp.\ \textbf{partially positive definite (ppd)}) 
if the following two conditions hold:
\begin{enumerate} 
  \item $a_{ij}$ is specified if and only if $a_{ji}$ is specified and $a_{ij}=a_{ji}$.
  \item All fully specified principal minors of $A$ are psd (resp.\ pd). 
\end{enumerate}

It is well-known that a ppsd matrix $A(x)$ of the form as in Lemma \ref{psd-completion} below admits a psd completion.
(This follows from the fact that the corresponding graph is chordal, see e.g.\ \cite{GJSW84,Dan92,BW11}.) 
In the notation of Lemma \ref{psd-completion}, if $A(x_0)$, $x_0\in \RR$, is a psd Hankel matrix, 
then Corollary \ref{rank-theorem-2} implies that \eqref{assump-on-ranks} below holds.
Since we will need an additional information about the rank of the completion $A(x_0)$ and 
the explicit interval of all possible $x_0$ for our results, we give a proof of Lemma \ref{psd-completion} based on 
the use of generalized Schur complements assuming \eqref{assump-on-ranks} holds.

\begin{lemma}\label{psd-completion}
	Let 
		$$A(x):=\left(\begin{array}{ccc}
			A_1 & a & b \\
			a^T & \alpha & x\\
			b^T & x & \beta
			\end{array}\right)\in S_n$$ 
	be a partially positive semidefinite symmetric matrix,
	where $A_1\in S_{n-2}$, $a,b\in \RR^{n-2}$, $\alpha,\beta\in \RR$ and $x$ is a variable.
	Let
		$$A_2:=\left(\begin{array}{cc}
				A_1 & a \\
				a^T & \alpha
				\end{array}\right)\in S_{n-1},\qquad
		A_3:=\left(\begin{array}{cc}
				A_1 & b \\
				b^T & \beta
				\end{array}\right)\in S_{n-1},$$
	and
		$$x_{\pm}:=b^TA_1^{+}a\pm \sqrt{(A_2/A_1)(A_3/A_1)}\in \RR.$$
	Suppose the following holds:
		\begin{equation}\label{assump-on-ranks}
			A_1\;\text{is invertible}\quad \text{or} \quad \Rank A_1=\Rank A_2.
		\end{equation}
	Then: 
	\begin{enumerate}
		\item\label{psd-comp-pt1} $A(x_{0})$ is positive semidefinite if and only if $x_0\in [x_-,x_+]$. 
	 	\item\label{psd-comp-pt2} 
			$$\Rank A(x_0)=
			\left\{\begin{array}{rl}
			\max\big\{\Rank A_2, \Rank A_3\big\},& \text{for}\;x_0\in \{x_-,x_+\},\\
			\max\big\{\Rank A_2, \Rank A_3\big\}+1,& \text{for}\;x_0\in (x_-,x_+).
			\end{array}\right.$$
		\item\label{psd-comp-pt3} If $A(x)$ is partially positive definite, then 
			$A(x')$ is positive definite for $x'\in (x_{-},x_{+})$.
	\end{enumerate}
\end{lemma}

\begin{proof}
By Theorem \ref{block-psd}, $A(x)\succeq 0$ if and only if
\begin{equation}\label{psd-cond-1}
	A_2\succeq 0, \qquad 
	\left(\begin{array}{c} b \\ x\end{array}\right)\in
	\cC(A_2)
	\qquad\text{and}
	\qquad
	f(x):= A(x)/A_2\geq 0,
\end{equation}
The first condition of \eqref{psd-cond-1} is true by the ppsd assumption.

Since $A_2\succeq 0$, it follows by Theorem \ref{block-psd} that $a\in \cC(A_1)$ and hence by the properties of the 
Moore-Penrose inverse we have that $A_1(A_1^+a)=a$.
Thus, 
	\begin{equation}\label{col-space-A2}
		\cC(A_2)=
		\cC\big(\left(\begin{array}{cc}
			A_1 & 0 \\
			a^T & \alpha-a^TA_1^+a
		\end{array}\right)\big)=
	\cC\big(\left(\begin{array}{cc}
			A_1 & 0 \\
			a^T & A_2/A_1
		\end{array}\right)\big).
	\end{equation}
Now we separate two cases according to $A_2/A_1$.\\

\noindent \textbf{Case 1:} $A_2/A_1>0$.\\

\eqref{col-space-A2} and the assumption of Case 1 imply that $\cC(A_2)=\cC(A_1\oplus 1)$.
Since $A_3\succeq 0$, it follows by Theorem \ref{block-psd} that $b\in \cC(A_1)$.
Therefore
	$\left(\begin{array}{cc} b & x\end{array}\right)^T \in \cC(A_1\oplus 1)$ 
for every $x\in \RR$. 
Thus the second condition of \eqref{psd-cond-1} is true for every $x\in \RR$.
	
Note that the assumption of Case 1 and Proposition \ref{rank-13-07} imply that $\Rank A_2>\Rank A_1$ and hence the assumption
\eqref{assump-on-ranks} implies invertibility of $A_1$ and $A_2$. 
By Proposition \ref{Schur3by3-ver2}, used for $A(x)$ as $K$, $A_2$ as $M$ and $A_1$ as $A$, we have that
	\begin{equation}\label{cond2}
		f(x)=A_3/A_1 -
		(A_2/A_1)^{-1}(x-b^TA_1^{+}a)^2.
	\end{equation} 
Therefore $f(x_0)\geq 0$ if and only if $x_0\in [x_-,x_+]$, which is the third condition of \eqref{psd-cond-1}.
Now by Proposition \ref{rank-13-07} we know that $\Rank A(x)>\Rank A_2$ if and only if $f(x_0)>0$, 
which establishes \eqref{psd-comp-pt1},\eqref{psd-comp-pt2} in the case $A_2/A_1>0$.
\\

\noindent \textbf{Case 2:} $A_2/A_1=0$.\\

\eqref{col-space-A2} and the assumption of Case 2 imply that 
	\begin{equation}\label{col-space-10-07-20}
		\cC(A_2)=\cC\big(\left(\begin{array}{c} A_1 \\ a^T\end{array}\right)\big).
	\end{equation}
Therefore, using \eqref{col-space-10-07-20}, it is true that
	\begin{equation}\label{col-space2-10-07-20}
		\left(\begin{array}{c} b \\ x\end{array}\right)\in \cC(A_2) \;\Leftrightarrow \;	
		\left(\begin{array}{c} b \\ x\end{array}\right)=
		\left(\begin{array}{c} A_1 \\ a^T\end{array}\right)z=
		\left(\begin{array}{c} A_1z \\ a^Tz\end{array}\right)\quad \text{for some }z\in \RR^{n-2}.
	\end{equation}
Since $A_3\succeq 0$, it follows by Theorem \ref{block-psd} that 
$b\in \cC(A_1)$ and hence by the properties of the 
Moore-Penrose inverse
	$\{A_1^+b+w\colon w\in \ker A_1\}$
is the set of all solutions $z$ of the system $A_1z=b$. 
Therefore, using \eqref{col-space2-10-07-20}, it follows that
	$$\left(\begin{array}{c} b \\ x\end{array}\right)\in \cC(A_2) \;\Leftrightarrow \;	
		x\in \{a^TA_1^+b+a^Tw\colon w\in \ker A_1\}=\{a^TA_1^+b\},$$
where we used the fact that $A_1$ is symmetric, $a\in \cC(A_1)$ and $w\in \ker A_1$ for the last equality.
So only $x_0=a^TA_1^+b$ satisfies the second condition of \eqref{psd-cond-1}.

Now by definition of the generalized Schur complement, we have
	$$f(x)=\beta-
	\left(\begin{array}{cc} b^T & x\end{array}\right)
	A_2^+
	\left(\begin{array}{c} b \\ x\end{array}\right).
	$$
By the properties of the Moore-Penrose inverse 
	$$A_2^+\left(\begin{array}{c} b \\ x_0\end{array}\right)=
	\left(\begin{array}{c} A_1^+b \\ 0\end{array}\right)+v\quad \text{for some}\;v\in \ker A_2.
	$$
Hence,
	$$f(x_0)
	  =\beta-
	\left(\begin{array}{cc} b^T & x_0\end{array}\right)
	\Big(\left(\begin{array}{c} A_1^+b \\ 0\end{array}\right)+v\Big)
       =\beta-b^TA_1^+b=A_3/A_1\geq 0,$$
where the second equality follows from
the fact that $A_2$ is symmetric, $\left(\begin{array}{cc} b^T & x_0\end{array}\right)^T\in \cC(A_2)$ and $v\in \ker A_2$, and
the last inequality follows by the ppsd assumption. 
Note that $x_0=x_+=x_-$ and by Proposition \ref{rank-13-07}, $\Rank A(x_0)=\Rank A_2$ if and only if $A_3/A_1=0$, 
in which case also $\Rank A_3=\Rank A_2$.
Otherwise we have $f(x_0)=A_3/A_1>0$, which implies by Proposition \ref{rank-13-07} that 
	$\Rank A(x_0)=\Rank A_3=\Rank A_1+1$.
Thus \eqref{psd-comp-pt1},\eqref{psd-comp-pt2} are true in the case $A_2/A_1=0$.

\eqref{psd-comp-pt3} follows from \eqref{psd-comp-pt2} by noticing that $A_2/A_1>0$, $A_3/A_1>0$ and $\Rank A_2=\Rank A_3=n-1.$
\end{proof}

\subsection{Extension principle}\label{SubS2.5}

The extension principle for psd matrices is the following.

\begin{lemma}\label{extension-principle}
	Let $A\in S_n$ be a positive semidefinite matrix, 
	$Q\subseteq \{1,\ldots,n\}$ a subset 
	and	
	$A_Q$ be the restriction of $A$ to rows and columns from the set $Q$. 
	If $v\in \ker A_Q$ is a nonzero vector from the kernel of $A_Q$,
	then the vector $\widehat{v}$ with the only nonzero entries in rows from $Q$ and such that the restriction 
	$\widehat{v}|_Q$ to the rows from $Q$ equals to $v$, belongs to $\ker A$.  
\end{lemma}

\begin{proof}
	By permuting rows and columns we may assume that $A$ is of the form
	   $A=\begin{mpmatrix} A_Q & B \\ B^T & C\end{mpmatrix}.$
	We have to prove that 
	\begin{equation}\label{kernel-14-07-20}
		A\begin{mpmatrix} v\\ 0 \end{mpmatrix}=0.
	\end{equation} 
	Since $A$ is psd, for every 
 	$w:=\left(\begin{array}{cc} v^T & u^T \end{array}\right)\in \RR^{n}$ we have that
	\begin{equation}\label{ineq-14-07-20}
	   0\leq w A w^T=2 u^T B^Tv+u^TCu.
	\end{equation}
	If $B^Tv\neq 0$, then we define $u:=-\alpha B^Tv$ where $\alpha>0$ is an arbitrary positive real number, and 
	plug into \eqref{ineq-14-07-20} to get
	\begin{equation}\label{kernel-psd-submatrix}
		0\leq -2\alpha \left\|B^Tv\right\|^2+\alpha^2 v^TBCB^Tv=\alpha(\alpha v^TBCB^Tv-2 \left\|B^Tv\right\|^2)
		=:\alpha S(\alpha).
	\end{equation}
	Since $\lim_{\alpha\to 0}S(\alpha)=-2 \left\|B^Tv\right\|^2<0$, \eqref{kernel-psd-submatrix} cannot be true for $\alpha$ small enough. 
	Hence $B^Tv= 0$, which proves \eqref{kernel-14-07-20}.
\end{proof}

\subsection{Subsequences of one-dimensional moment sequences}\label{SubS2.6}

\begin{proposition}\label{measure-for-subsequence}
	Let $k\in \NN$ and $\beta=(\beta_0,\ldots,\beta_{2k})$ with $\beta_0>0$ be a sequence which admits a representing measure supported on $K=\RR$.
	Then for every $i,j\in \NN$, where $0\leq i\leq j\leq k$, a subsequence $\beta^{(i,j)}:=(\beta_{2i},\ldots,\beta_{2j})$ also admits a representing measure supported on $K=\RR$.
\end{proposition}

\begin{proof}
	Note that $A_\beta$ is of the form
		$$A_{\beta}=
			\left(\begin{array}{ccc} 
				A_{\beta^{(0,i-1)}} & \ast & \ast \\ 
				\ast & A_{\beta^{(i,j)}} & \ast \\ 
				\ast & \ast & A_{\beta^{(j+1,k)}}
			\end{array}\right).$$
	By Theorem \ref{Hamburger}, $A_\beta\succeq 0$ and hence $A_{\beta^{(i,j)}}\succeq 0$. 
	For $i=j$ the statement is clear, i.e., the representing atom is $\beta_{2i}$ with density 1. 
	Assume that $i<j$. We separate two cases according to the invertibility of $A_{\beta^{(i,j)}}$.
	\begin{enumerate}
		\item If $A_{\beta^{(i,j)}}\succ 0$, then 
				$\Rank A_{\beta^{(i,j)}}=\Rank \beta^{(i,j)}=j-i+1$
			and by Theorem \ref{Hamburger}, $\beta^{(i,j)}$ admits a measure.
		\item\label{case-2-subsequence-14-07-20}
			Else 
				$$A_{\beta^{(i,j)}}=\left(\begin{array}{cc} A_{\beta^{(i,j-1)}} & v^T \\ v & \beta_{2j} \end{array}\right)$$ 
			is singular, where 
				$v=\left(\begin{array}{ccc}\beta_{j} & \cdots & \beta_{2j-1}\end{array}\right)$. 
			We separate two cases according to the invertibility of $A_{\beta^{(i,j-1)}}$.
			\begin{itemize}
					\item If $A_{\beta^{(i,j-1)}}$ is invertible, then $\Rank A_{\beta^{(i,j-1)}}=\Rank A_{\beta^{(i,j)}}$.
					\item Else
					 	$A_{\beta^{(i,j-1)}}$ is singular and by Corollary \ref{rank-theorem-2} used for $\beta^{(i,j)}$ as $\beta$,
						we get 
							$\Rank A_{\beta^{(i,j-2)}}=\Rank A_{\beta^{(i,j-1)}}$.
						This implies that the last column of $A_{\beta^{(i,j-1)}}$ is in the span of the other columns of $A_{\beta^{(i,j-1)}}$.
						By Lemma \ref{extension-principle}, the $j$-th column of $A_{\beta}$ is in the span of the columns $i+1,\ldots, j-1$.  
						Since $\beta$ is positively recursively generated, the $(j+1)$-th column of $A_{\beta}$ is in the span of the columns $i+2,\ldots, j$
						and in particular the last column of  $A_{\beta^{(i,j)}}$ is in the span of the other columns of $A_{\beta^{(i,j)}}$.
						Hence $\Rank A_{\beta^{(i,j-1)}}=\Rank A_{\beta^{(i,j)}}$.
			\end{itemize}
		In both subcases of \eqref{case-2-subsequence-14-07-20}, $\Rank A_{\beta^{(i,j-1)}}=\Rank A_{\beta^{(i,j)}}$ and
		Corollary \ref{singular-case-measure} implies that $\beta^{(i,j)}$ admits a measure.
	 \end{enumerate}
\end{proof}

\section{Truncated Hamburger moment problem of degree $2k$ with gap $(\beta_{2k-1})$ and
	 $(\beta_{2k-2},\beta_{2k-1})$}
\label{S3}

In this section we solve the THMP of degree $2k$ with gaps $(\beta_{2k-1})$ 
	(see Theorem \ref{trunc-Hamb-without-3n-1-bg})
and $(\beta_{2k-2},\beta_{2k-1})$ 
	(see Theorem \ref{trunc-Hamb-without-3n-2-and-1}). 
As a corollary of Theorem \ref{trunc-Hamb-without-3n-1-bg} 
we obtain the solution to the TMP for the curve $y=x^3$ (see Corollary \ref{Y=X3-gen-bg}),
while as a corollary of Theorem \ref{trunc-Hamb-without-3n-2-and-1} 
we get the solution to the TMP for the curve $y=x^4$ and an additional moment $\beta_{3,2k-2}$ given
(see Corollary \ref{Y=X4-gen}).

\subsection{Truncated Hamburger moment problem 
of degree $2k$ with gap $(\beta_{2k-1})$}

\begin{theorem}
\label{trunc-Hamb-without-3n-1-bg} 
  Let $k\in \NN$ and 
	$$\beta(x):=(\beta_0,\beta_1,\ldots,\beta_{2k-2},x,\beta_{2k})$$ 
  be a sequence where each $\beta_i$ is a real number, $\beta_0>0$ and $x$ is a variable. 
  Let 
	$$\widehat{\beta}:=(\beta_0,\ldots,\beta_{2k-4})
		\quad{and}\quad 
		\widetilde{\beta}:=(\beta_0,\ldots,\beta_{2k-2})$$
  be subsequences of $\beta(x)$, 
	$v:=(\begin{array}{ccc}\beta_{k} & \cdots & \beta_{2k-2}\end{array})$ a vector 
  and 
	$$\widetilde{A}:=\left(\begin{array}{cc} A_{\widehat\beta} & v^T \\ v & \beta_{2k}\end{array}\right)$$
  a matrix.
  Then the following statements are equivalent: 
\begin{enumerate}	
	\item\label{pt1-v1206} There exists $x_0\in \RR$ and a representing measure for $\beta(x_0)$ supported on $K=\RR$.
	\item\label{pt2-v1206} There exists $x_0\in \RR$ and a $(\Rank \widetilde\beta)$-atomic representing measure for $\beta(x_0)$.
	\item\label{pt3-v1206} $A_{\beta(x)}$ is partially positive semidefinite and one of the following conditions is true: 
	\begin{enumerate}
		\item\label{k1-11:52-200720} $k=1$.
		\item $k>1$ and one of the following conditions is true:
			\begin{enumerate}
				\item\label{pt3a-v1706} $A_{\widetilde{\beta}}\succ 0$.
				\item\label{pt3b-v1706} 
						$\Rank A_{\widehat\beta}=\Rank A_{\widetilde{\beta}}=\Rank \widetilde A.$
			\end{enumerate}
	\end{enumerate}
\end{enumerate}
\end{theorem}

\begin{proof}
First we prove the implication $\eqref{pt1-v1206}\Rightarrow\eqref{pt3-v1206}$.
By Theorem \ref{Hamburger}, $A_{\beta(x_0)}\succeq 0$ and $\Rank A_{\beta(x_0)}=\Rank\beta(x_0)$.
$A_{\beta(x_0)}\succeq 0$ in particular implies that $A_{\beta(x)}$ is ppsd. 
If $k=1$, then \eqref{k1-11:52-200720} holds.
Otherwise $k>1$.
If $A_{\widetilde\beta}\succ 0$, then \eqref{pt3a-v1706} holds.
Else $A_{\widetilde{\beta}}$ is singular and hence
	\begin{equation}\label{rank-eq-14-07-20}
		\Rank A_{\widehat{\beta}}=\Rank A_{\widetilde{\beta}}=\Rank \beta(x_0)=A_{\beta(x_0)},
	\end{equation}
where the first two equalities follow by Corollary \ref{rank-theorem-2} 
used for $\beta(x_0)$ as $\beta$ and the last by Theorem \ref{Hamburger}.  
$A_{\widehat\beta}$ being a principal submatrix of  $\widetilde{A}$ and 
$\widetilde{A}$ being a principal submatrix of  
	$$A_{\beta(x_0)}=
		\left(\begin{array}{ccc}
			A_{\widehat \beta} & u^T & v^T\\
			u & \beta_{2k-2} & x_0\\
			v & x_0 & \beta_{2k}
		\end{array}\right),$$ 
where $u=(\begin{array}{ccc}\beta_{k-1} & \cdots & \beta_{2k-3}\end{array})$,
imply together with \eqref{rank-eq-14-07-20} that \eqref{pt3b-v1706} holds
and concludes the proof of the implication  $\eqref{pt1-v1206}\Rightarrow\eqref{pt3-v1206}$.

Second we prove the implication $\eqref{pt3-v1206}\Rightarrow\eqref{pt2-v1206}$.
We separate two cases according to $k$.
\begin{itemize}
	\item $k=1$. We have that 
		$A_{\beta(x)}=\left(\begin{array}{cc} \beta_0 & x \\ x & \beta_{2}\end{array}\right)$.
		For $x_0=\sqrt{\beta_0\beta_2}$, $A_{\beta(x_0)}$ is of rank 1 and the second column is the multiple
		of the first. Hence, by Corollary \ref{singular-case-measure}, a 1-atomic measure exists, proving 
		the implication $\eqref{pt3-v1206}\Rightarrow\eqref{pt2-v1206}$ in this case.
	\item $k>1$. Notice that $A_{\beta(x)}$ is of the same form as $A(x)$ from Lemma \ref{psd-completion},
		where $A_{\widehat \beta}$, $A_{\widetilde \beta}$, $\widetilde A$ correspond to $A_1$, $A_2$, $A_3$, 		
		respectively.
		Since both cases \eqref{pt3a-v1706} and \eqref{pt3b-v1706} satisfy the assumption 
		\eqref{assump-on-ranks},
		it follows by Lemma \ref{psd-completion} that there exists $x_0$ such that $A_{\beta(x_0)}\succeq 0$ and
		\begin{equation}\label{rank-psd-lemma}
			\Rank A_{\beta(x_0)}=
				\max\left\{\Rank A_{\widetilde{\beta}},
				\Rank \widetilde{A}\right\}.
		\end{equation}
		Since in the case \eqref{pt3a-v1706}, it holds that $\Rank \widetilde{A}\leq \Rank A_{\widetilde{\beta}}$, 
		while in the case \eqref{pt3b-v1706},
		$\Rank \widetilde{A}= \Rank A_{\widetilde{\beta}}$, we obtain from \eqref{rank-psd-lemma} that
				$\Rank A_{\beta(x_0)}=\Rank A_{\widetilde{\beta}}.$
		By Corollary \ref{singular-case-measure}, $(\Rank \widetilde\beta)$-representing measure for $\beta(x_0)$ 	
		exists, which proves $\eqref{pt2-v1206}$.
\end{itemize}

The implication $\eqref{pt2-v1206}\Rightarrow\eqref{pt1-v1206}$ is trivial.
\end{proof}

\begin{example}\label{ex-v1}
	For $k=9$, let
		\begin{align*}
			\beta^{(1)}(x) &=(1,0,1,0,2,0,5,0,14,0,42,0,132,0,429,0,2000,x,338881),\\
			\beta^{(2)}(x) &=\Big(14,\frac{7}{2},\frac{79}{4},-\frac{67}{8},\frac{1055}{16},-\frac{1935}{32},
						\frac{18195}{64},-\frac{43115}{128},\frac{336151}{256},-\frac{926695}{512},
						\frac{6407195}{1024},-\frac{19736547}{2048},\\
			&\hspace{0,5cm}\frac{124731423}{4096},-\frac{419176415}{8192},\frac{2469281827}{16384},-\frac{8894873563}{32768},
				\frac{49568350247}{65536},x,\frac{1006568996907}{262144}\Big),\\
			\beta^{(3)}(x) &= (8,0,78,0,1446,0,32838,0,794886,0,19651398,0,489352326,0,12216629958,0,305262005766,\\
			&\hspace{0,5cm}x,7630169896518).
		\end{align*}
	Let $\widetilde A^{(i)}$, $i=1,2,3$, denote $\widetilde A$ from Theorem \ref{trunc-Hamb-without-3n-1-bg} 
	corresponding to $\beta^{(i)}(x)$. Using \textit{Mathematica} \cite{Wol} one can check that:
	\begin{itemize}
		\item $\widetilde A^{(i)}\succeq 0$ for $i=1,2,3.$
		\item $A_{\widetilde \beta^{(1)}}\succ 0$, 
				$A_{\widetilde \beta^{(2)}}\not\succeq 0$, 
				$A_{\widetilde \beta^{(3)}}\succeq 0$ and $\dim\left(\ker A_{\widetilde \beta^{(3)}}\right)=1$.
		\item $\Rank A_{\widehat \beta^{(3)}}=\Rank \widetilde A^{(3)}=\Rank A_{\widetilde \beta^{(3)}}=8$.
	\end{itemize}
	Therefore:
	\begin{itemize}
		\item  $A_{\beta^{(1)}(x)}$ is ppsd and $\widetilde \beta^{(1)}$ satisfies \eqref{pt3a-v1706} of 		
			Theorem \ref{trunc-Hamb-without-3n-1-bg}, implying that a 9-atomic measure for 
			$\beta^{(1)}(x)$ exists.
		\item $A_{\beta^{(2)}(x)}$ is not ppsd and by Theorem \ref{trunc-Hamb-without-3n-1-bg}, there is no
			representing measure for $\beta^{(2)}(x)$.
		\item $A_{\beta^{(3)}(x)}$ is ppsd and $\widetilde \beta^{(3)}$ satisfies \eqref{pt3b-v1706} of 		
			Theorem \ref{trunc-Hamb-without-3n-1-bg}, implying that an 8-atomic measure for 
			$\beta^{(3)}(x)$ exists.
	\end{itemize}
\end{example}

The following corollary is a consequence of Theorem \ref{trunc-Hamb-without-3n-1-bg} and is an alternative solution of the bivariate TMP 
for the curve $y=x^3$, first solved by Fialkow in \cite{Fia11}. 

\begin{corollary}
	\label{Y=X3-gen-bg}
	Let $k\in \NN$ and
 		$\displaystyle\beta=(\beta_{i,j})_{i,j\in \ZZ^2_+,i+j\leq 2k}$ 
	be a 2-dimensional real multisequence of degree $2k$.
	Suppose $M(k)$ is positive semidefinite and recursively generated.
	Let 
	$$u^{(i)}:=(\beta_{0,i},\beta_{1,i},\beta_{2,i})\quad \text{for }i=0,\ldots,2k-2,$$
	$$\widehat{\beta}:=(u^{(0)},\ldots,u^{(2k-2)})\qquad\text{and}\qquad
		\widetilde{\beta}:=(u^{(0)},\ldots,u^{(2k-2)},\beta_{0,2k-1},\beta_{1,2k-1})$$
	be subsequences of $\beta$.
	Then $\beta$ has a representing measure supported on $y=x^3$ if and only if the following statements hold:
	\begin{enumerate}
		\item One of the following holds: 
			\begin{itemize}
				\item If $k\geq 3$, then $Y=X^3$ is a column relation of $M(k)$. 
				\item If $k=2$, then the equalities 
					$\beta_{0,1}=\beta_{3,0}$, $\beta_{1,1}=\beta_{4,0}$, $\beta_{0,2}=\beta_{3,1}$ hold.
			\end{itemize}
		\item\label{point1-121620} One of the following holds: 
			\begin{enumerate}
				\item $A_{\widetilde\beta}\succ 0$.	
				\item $A_{\widetilde\beta}\succeq 0$ and 
					$\Rank A_{\widehat\beta}=\Rank A_{\widetilde\beta}=\Rank M(k)$.
			\end{enumerate}
	\end{enumerate}

	Moreover, if the representing measure exists, then:
	\begin{itemize}
		\item If $A_{\widetilde\beta}$ is nonsingular, there exists a $(3k)$-atomic measure. 
		\item If $A_{\widetilde\beta}$ is singular, then the measure is
			$(\Rank M(k))$-atomic.
	\end{itemize}
\end{corollary}

\begin{proof}
	For $m\in \{0,1\ldots,6k-2,6k\}$ we define the numbers $\widetilde \beta_m$ by the following rule
		\begin{equation*}
			\widetilde \beta_m:=\beta_{m\Mod{3},\lfloor\frac{m}{3}\rfloor}.
		\end{equation*}
	\noindent \textbf{Claim 1.} Every number $\widetilde \beta_m$ is well-defined.\\

	We have to prove that $m\Mod{3}+\lfloor\frac{m}{3}\rfloor\leq 2k.$ 
	We separate three cases according to $m$.
	\begin{itemize}
		\item $m\leq 6k-4$: $\lfloor\frac{m}{3}\rfloor+m\Mod{3}\leq (2k-2)+2=2k$.
		\item $m\in \{6k-3,6k-2\}$: 
			$\lfloor\frac{m}{3}\rfloor+m\Mod{3}\leq (2k-1)+1=2k$.
		\item $m=6k$: $\lfloor\frac{m}{3}\rfloor+m\Mod{3}= 2k+0=2k.$\\
	\end{itemize}
	
	\noindent \textbf{Claim 2.} 
		Let $t\in \NN$. 
		The atoms $(x_1,x_1^3),\ldots (x_t,x_t^3)$ with densities $\lambda_1,\ldots,\lambda_t$
		are the $(y-x^3)$-representing measure for $\beta$
		if and only if
		the atoms $x_1,\ldots,x_t$ with densities $\lambda_1,\ldots,\lambda_t$
		are the $\RR$-representing measure for 
		 $\widetilde \beta(x)=(\widetilde \beta_0,\ldots,\widetilde \beta_{2k-2},x,\widetilde\beta_{2k})$.\\

	The if part follows from the following calculation:
		$$\widetilde \beta_{m}
			=\beta_{m\Mod{3},\lfloor\frac{m}{3}\rfloor}
			=\sum_{\ell=1}^t \lambda_\ell x_\ell^{m\Mod{3}}x_{\ell}^{3\lfloor\frac{m}{3}\rfloor}
			=\sum_{\ell=1}^t \lambda_\ell x_\ell^{m\Mod{3}+3\lfloor\frac{m}{3}\rfloor}
			=\sum_{\ell=1}^t \lambda_\ell x_\ell^{m},$$
	where $m=0,1,\ldots,6k-2,6k$.

	The only if part follows from the following calculation:
	\begin{align*}
	  \beta_{i,j}
		&= \beta_{i-3,j+1}=\cdots=\beta_{i\Mod{3},j+\lfloor\frac{i}{3}\rfloor}\\
		&=\widetilde \beta_{i\Mod{3}+3(j+\lfloor\frac{i}{3}\rfloor)}
		  = \sum_{\ell=1}^t \lambda_\ell x_\ell^{i\Mod{3}+3(j+\lfloor\frac{i}{3}\rfloor)}
		  =\sum_{\ell=1}^t \lambda_\ell x_\ell^{i\Mod{3}+3\lfloor\frac{i}{3}\rfloor} x_{\ell}^{3j}
		  =\sum_{\ell=1}^t \lambda_\ell x_\ell^{i}(x_{\ell}^3)^{j},
	\end{align*}
	where the equalities in the first line follow by $M(k)$ being rg.\\

	Using Claim 2 and a theorem of Bayer and Teichmann \cite{BT06}, implying that if a finite sequence 
	has a $K$-representing measure, then it has a finitely atomic $K$-representing measure, the statement of the 	
	Corollary follows by Theorem \ref{trunc-Hamb-without-3n-1-bg}.
\end{proof}

\begin{remark}
	\begin{enumerate}
		\item Corollary \ref{Y=X3-gen-bg} in case $k=1$ is an improvement of \cite[Proposition 5.6.ii)]{Fia11} by decreasing the number of 
			atoms from 6 to 3. 
		\item For $M(1)\succ 0$ and $A_{\widetilde\beta}\not\succ 0$, \eqref{point1-121620} of Corollary
			 \ref{Y=X3-gen-bg} is not satisfied and hence the measure does not exist.
			Since this is the case under the assumptions of \cite[Proposition 5.6.iii)]{Fia11},
			the additional conditions in \cite[Proposition 5.6.iii)]{Fia11} are never satisfied.
		\item Examples in the Example \ref{ex-v1} above are derived from \cite[Example 5.2]{Fia11}, \cite[Example 4.18]{Fia08}, \cite[Example 3.3]{Fia08}, 
		which demonstrate the solution of the moment problem for the curve $y=x^3$.
	\end{enumerate}
\end{remark}

\subsection{Truncated Hamburger moment problem of degree $2k$ with gaps $(\beta_{2k-2},\beta_{2k-1})$}
%

\begin{theorem}\label{trunc-Hamb-without-3n-2-and-1} 
  Let $k\in \NN$, $k>1$, and 
	$$\beta(x,y):=(\beta_0,\beta_1,\ldots,\beta_{2k-3},y,x,\beta_{2k})$$ 
  be a sequence, where each $\beta_i$ is a real number, $\beta_0>0$ and $x,y$ are variables.   
  Let 
	$$\widehat{\beta}:=(\beta_0,\ldots,\beta_{2k-6})\quad \text{and}\quad \widetilde{\beta}:=(\beta_0,\ldots,\beta_{2k-4})$$
  be subsequences of $\beta(x,y)$,
	$$
	  u:=\left(\begin{array}{ccc}\beta_{k} & \cdots & \beta_{2k-3}\end{array}\right),\quad
	  s:=\left(\begin{array}{ccc}\beta_{k-1} & \cdots & \beta_{2k-3}\end{array}\right)\quad \text{and}\quad
	  w:=\left(\begin{array}{ccc}\beta_{k-2} & \cdots & \beta_{2k-5}\end{array}\right)
	$$
  vectors and 
  	$$\widetilde{A}:=\left(\begin{array}{cc} A_{\widehat\beta} & u^T \\ u & 
		\beta_{2k}\end{array}\right)$$
  a matrix.
  Then the following statements are equivalent: 
\begin{enumerate}	
	\item\label{pt1-v1606} There exist $x_0,y_0\in \RR$ and a representing measure for $\beta(x_0,y_0)$ supported on $K=\RR$.
	\item\label{pt2-v1606} There exist $x_0,y_0\in \RR$ and a 
		$(\Rank \widetilde\beta)$ or 
		$(\Rank \widetilde\beta+1)$-atomic representing measure for $\beta(x_0,y_0)$.
	\item\label{pt3-v1606} $A_{\beta(x,y)}$ is partially positive semidefinite and one of the following conditions 
		holds: 
	\begin{enumerate}
		\item\label{pt1-14:21-200720} $k=2$ and $\frac{\beta_1^2}{\beta_0}\leq \sqrt{\beta_0\beta_4}$. 
		\item $k>2$, the inequality
				\begin{equation}\label{cond-to-solve-v1606}
					 	sA_{\widetilde{\beta}}^{+}s^T \leq u A_{\widehat{\beta}}^{+}w^T+
						\sqrt{(A_{\widetilde{\beta}}/A_{\widehat{\beta}}) (\widetilde{A}/A_{\widehat{\beta}})}.
				\end{equation}
				holds and one of the following conditions is true: 
			\begin{enumerate}
				\item\label{case1} $A_{\widetilde{\beta}}\succ 0$.
				\item\label{case2} 
				$\Rank A_{\widehat{\beta}}=\Rank A_{\widetilde{\beta}}=
				\Rank \left(\begin{array}{cc}A_{\widetilde{\beta}}& s^T\end{array}\right)=\Rank \widetilde{A}$.
				\end{enumerate}
	\end{enumerate}
\end{enumerate}

Moreover, if the representing measure for $\beta$ exists, then:
\begin{itemize}
	\item If $k=2$, then there is a $1$-atomic measure if $\frac{\beta_1^2}{\beta_0}= \sqrt{\beta_0\beta_4}$.
	  Otherwise there is a 2-atomic measure. 
	\item If $k>2$, there exists a $(\Rank \widetilde\beta)$-atomic if and only if one of the equalities
	\begin{equation}\label{cond-to-solve-v1606-minimally}
		sA_{\widetilde{\beta}}^{+}s^T = 
			u A_{\widehat{\beta}}^{+}w^T- \sqrt{(A_{\widetilde{\beta}}/
			A_{\widehat{\beta}}) (\widetilde{A}/A_{\widehat{\beta}})}\quad \text{or}\quad
		sA_{\widetilde{\beta}}^{+}s^T = 
			u A_{\widehat{\beta}}^{+}w^T+ \sqrt{(A_{\widetilde{\beta}}/
			A_{\widehat{\beta}}) (\widetilde{A}/A_{\widehat{\beta}})}
	\end{equation}
	holds.
\end{itemize}
\end{theorem}

\begin{proof}
Note that $\beta(x,y)$ admits a measure if and only if there exist $y_0\in \RR$
such that $\beta(x,y_0)$ admits a measure.
Theorem \ref{trunc-Hamb-without-3n-1-bg} implies that the following claim holds.\\

\noindent\textbf{Claim 1.} 
$\beta(x,y_0)$ admits a measure if and only if the following conditions hold: 
\begin{enumerate}
	\item\label{pt1-C1-15-07-20} $A_{\beta(x,y_0)}$ is ppsd.
	\item One of the following is true:
	\begin{enumerate}
		\item\label{cond1-proof-10-07} $A_{(\widetilde\beta,\beta_{2k-3},y_0)}\succ 0$, where
			$$A_{(\widetilde\beta,\beta_{2k-3},y)}=
			\left\{\begin{array}{ll}
				\left(\begin{array}{cc} \beta_0 & \beta_1 \\ \beta_1 & y \end{array}\right),&
						\text{if }k=2,\\
				\left(\begin{array}{cc} A_{\widetilde{\beta}} & s^T \\ s & y_0 \end{array}\right)=
				\left(\begin{array}{ccc} 
					A_{\widehat{\beta}} & w^T & s_1^T \\ 
					w & \beta_{2k-4} & \beta_{2k-3} \\ 
					s_1 & \beta_{2k-3} & y_0
				\end{array}\right)\quad
			\text{where}\quad s_1^T=
			\left(\begin{array}{c} \beta_{k-1} \\ \vdots \\ 
								\beta_{2k-4} \end{array}\right),&	\text{otherwise}.\end{array}\right.$$ 
		\item\label{cond2-proof} 
			$\Rank A_{\widetilde\beta}=\Rank \widehat A(y_0),$
		where  
			$$\widehat A(y):=\left\{\begin{array}{ll}
				\left(\begin{array}{cc} \beta_0 & y \\ y & \beta_{4}\end{array}\right),& \text{if }k=2,\\
				\left(\begin{array}{cc} A_{\widetilde{\beta}} &  u(y)^T \\ u(y) & \beta_{2k}\end{array}\right)=
				\left(\begin{array}{ccc} A_{\widehat{\beta}} & w^T & u^T \\ w & \beta_{2k-4} & y \\ u & y & 							\beta_{2k}\end{array}\right)
				\quad 
				\text{and}\quad
			  	u(y):=\left(\begin{array}{cc} u &y \end{array}\right),& \text{otherwise}.
				\end{array}\right.$$
	\end{enumerate}
\end{enumerate}

\noindent\textbf{Claim 2.} 
	Let $k>2$. 
	Assume $A_{\widehat \beta}\succ 0$ or $\Rank A_{\widehat \beta}=\Rank A_{\widetilde \beta}$.
	Then $\widehat A(y_0)\succeq 0$ if and only if 
	\begin{equation}\label{psd-cond-1-case3}
		\widehat A(y)\;\text{is ppsd}\quad \text{and}\quad 
		y_0\in 
		\Big[u A_{\widehat \beta}^+w^T-\sqrt{(A_{\widetilde{\beta}}/A_{\widehat{\beta}}) (\widetilde A/A_{\widehat \beta})},
			u A_{\widehat \beta}^+w^T+\sqrt{(A_{\widetilde{\beta}}/A_{\widehat{\beta}})(\widetilde A/A_{\widehat \beta})}\Big]=:[y_-,y_+].
	\end{equation}
	Moreover,
	\begin{equation}\label{rank-cond-15-07-20}
		\Rank \widehat A(y_0)=
			\left\{\begin{array}{rl}
			\max\big\{\Rank A_{\widetilde \beta}, \Rank \widetilde A\big\},& \text{for}\;y_0\in \{y_-,y_+\},\\
			\max\big\{\Rank A_{\widetilde \beta}, \Rank \widetilde A\big\}+1,& \text{for}\;y_0\in (y_-,y_+).
			\end{array}\right.
	\end{equation}\\

The assumption \eqref{assump-on-ranks} of Lemma \ref{psd-completion} used for 
$\widehat A(y), A_{\widehat \beta}, A_{\widetilde \beta}, \widetilde A$ as $A(x), A_1, A_2, A_3$, respectively, are by the assumption of Claim 2 satisfied and hence Claim 2 follows by Lemma \ref{psd-completion}.\\

\noindent\textbf{Claim 3.} 
	Let $k>2$. Assume $A_{\widehat \beta}\succ 0$ or $\Rank A_{\widehat \beta}=\Rank A_{\widetilde \beta}$.
	Then $A_{\beta(x,y_0)}$ is ppsd for some $y_0\in \RR$ 
		if and only if
	$A_{\beta(x,y)}$ is ppsd, $s^T\in\cC(A_{\widetilde\beta})$  and  \eqref{cond-to-solve-v1606} holds. \\

Note that 
	$A_{\beta(x,y_0)}$  is ppsd 
		if and only if 
	$A_{(\widetilde\beta,\beta_{2k-3},y_0)}\succeq 0$ and $\widehat{A}(y_0)\succeq 0$.
By Theorem \ref{block-psd}, 
$A_{(\widetilde\beta,\beta_{2k-3},y_0)}\succeq 0$ if and only if
\begin{equation}\label{psd-cond-1-case2}
	A_{\widetilde\beta}\succeq 0, \qquad 
	s^T\in\cC(A_{\widetilde\beta})
	\qquad\text{and}\qquad
	A_{(\widetilde\beta,\beta_{2k-3},y_0)}/A_{\widetilde\beta}=y_0-sA_{\widetilde\beta}^{+}s^T\geq 0,
\end{equation}
By Claim 2, $\widehat A(y_0)$ is psd if and only if \eqref{psd-cond-1-case3} holds.
Now note that the first condition of \eqref{psd-cond-1-case3} (which also includes the first condition of \eqref{psd-cond-1-case2}) is equivalent to 
$A_{\beta(x,y)}$ being ppsd and that $y_0$ satisfying the third condition of \eqref{psd-cond-1-case2}
and the second condition of \eqref{psd-cond-1-case3} exists if and only if  \eqref{cond-to-solve-v1606} holds.
This proves Claim 3.\\

First we prove the implication $\eqref{pt1-v1606}\Rightarrow\eqref{pt3-v1606}$.
By Claim 1, in particular $A_{\beta(x,y_0)}$ is ppsd. 

If $k=2$, then 
$A_{(\widetilde \beta,\beta_1,y)}\succeq 0$, which implies that $y_0\geq \frac{\beta_1^2}{\beta_0}$,
and $\widehat A(y_0)\succeq 0$, which implies that $y_0\leq \sqrt{\beta_0\beta_4}$.
Hence, $\frac{\beta_1^2}{\beta_0}\leq \sqrt{\beta_0\beta_4}$, which is \eqref{pt1-14:21-200720}.
Since $A_{\beta(x,y_0)}$ being ppsd implies that also $A_{\beta(x,y)}$ is ppsd, this proves the
implication $\eqref{pt1-v1606}\Rightarrow\eqref{pt3-v1606}$ in this case.

It remains to prove $\eqref{pt1-v1606}\Rightarrow\eqref{pt3-v1606}$ in the case $k>2$. 
We separate two cases according to the invertibility of $A_{\widetilde \beta}$.
\begin{itemize}
  \item $A_{\widetilde \beta}\succ 0$: Using Claim 3, $A_{\beta(x,y)}$ is ppsd, \eqref{cond-to-solve-v1606} and \eqref{case1} holds, which
	proves the implication $\eqref{pt1-v1606}\Rightarrow\eqref{pt3-v1606}$ in this case.
  \item $A_{\widetilde \beta}\not\succ 0$: It follows that $A_{(\widetilde\beta,\beta_{2k-3},y_0)} \not\succ 0$ 
	and hence \eqref{cond2-proof} of Claim 1 must hold. 
	Corollary \ref{rank-theorem-2} used for $(\widetilde\beta,\beta_{2k-3},y_0)$ as $\beta$ implies that
 	  \begin{equation}\label{rank1-15-07-20}
		\Rank A_{\widehat \beta}=\Rank A_{\widetilde \beta}.
	  \end{equation}
	By Proposition \ref{measure-for-subsequence}, 
		$(\widetilde\beta,\beta_{2k-3},y_0)$ also admits a measure and 
	Corollary \ref{singular-case-measure} used for $(\widetilde\beta,\beta_{2k-3},y_0)$ as $\beta$ implies that 
	  \begin{equation}\label{rank2-15-07-20}
		\Rank A_{\widetilde \beta}=\Rank A_{(\widetilde\beta,\beta_{2k-3},y_0)}.
	  \end{equation}
	\eqref{cond2-proof} of Claim 1 together with \eqref{rank1-15-07-20} implies that all the inequalities
	 in the estimate
		$\Rank A_{\widehat \beta}\leq \Rank \widetilde{A}\leq \Rank \widehat A(y_0)$ 
	are equalities and in particular,
 	  \begin{equation}\label{rank3-15-07-20}
		\Rank A_{\widehat \beta}= \Rank \widetilde{A}.
	  \end{equation}
	\eqref{rank1-15-07-20}, \eqref{rank2-15-07-20}, \eqref{rank3-15-07-20} and Claim 3 imply that 
	$A_{\beta(x,y)}$ is ppsd, \eqref{cond-to-solve-v1606} and \eqref{case2} holds,
	which proves the implication $\eqref{pt1-v1606}\Rightarrow\eqref{pt3-v1606}$ in this case.
\end{itemize}

Second we prove the implication $\eqref{pt3-v1606}\Rightarrow\eqref{pt1-v1606}$.
We separate two cases according to $k$.

If $k=2$, then we are in the case \eqref{pt1-14:21-200720}. For $y_0=\sqrt{\beta_0\beta_4}$, 
$\beta(x,y_0)$ is ppsd and satisfies \eqref{cond1-proof-10-07} of Claim 1 if 
$\frac{\beta_1^2}{\beta_0}< \sqrt{\beta_0\beta_4}$ and 
\eqref{cond2-proof} if $\frac{\beta_1^2}{\beta_0}=\sqrt{\beta_0\beta_4}$. In both cases 
Claim 1 implies the implication $\eqref{pt3-v1606}\Rightarrow\eqref{pt1-v1606}$ is true in this case.

Else $k>2$. If \eqref{case1} holds, then in particular 
	$A_{\widehat\beta}\succ 0$.
Otherwise \eqref{case2} holds and in particular 
	$\Rank A_{\widehat{\beta}}=\Rank A_{\widetilde{\beta}}$.
In both cases the assumptions of Claims 2 and 3 are fulfilled.
By Claim 3, the matrix $A_{\beta(x,y_+)}$ is ppsd and by \eqref{rank-cond-15-07-20} of Claim 2, 
$\Rank \widehat A(y_{+})=\max\{\Rank A_{\widetilde \beta}, \Rank \widetilde A\}$.
If \eqref{case1} holds, then $\Rank \widehat A(y_{+})=\Rank A_{\widetilde \beta}=k-1$.
Else \eqref{case2} holds and $\Rank \widehat A(y_{+})=\Rank A_{\widetilde \beta}=\Rank \widetilde A$.
In both cases, $\beta(x,y_+)$ satisfies \eqref{pt1-C1-15-07-20} and \eqref{cond2-proof} of Claim 1 above
and thus the measure exists which proves  the implication $\eqref{pt3-v1606}\Rightarrow\eqref{pt1-v1606}$.

The implication $\eqref{pt2-v1206}\Rightarrow\eqref{pt1-v1206}$ is trivial.

Now we prove the implication $\eqref{pt1-v1206}\Rightarrow\eqref{pt2-v1206}$. 
If $\beta(x,y_0)$ has a representing measure, then:
\begin{itemize}
  \item By Theorem \ref{trunc-Hamb-without-3n-1-bg} it has a $(\Rank (\widetilde\beta,\beta_{2k-3},y_0))$-atomic representing measure.
  \item By Proposition \ref{measure-for-subsequence}, $\widetilde \beta$ and $(\widetilde\beta,\beta_{2k-3},y_0)$ also have measures and 
	hence by Theorem \ref{Hamburger},
		$\Rank A_{\widetilde \beta}=\Rank \widetilde \beta$ 
	and
		$\Rank (\widetilde\beta,\beta_{2k-3},y_0)=\Rank A_{(\widetilde\beta,\beta_{2k-3},y_0)}.$
\end{itemize}
Since $\Rank A_{(\widetilde\beta,\beta_{2k-3},y_0)}\in \{\Rank A_{\widetilde \beta},\Rank A_{\widetilde \beta}+1\}$,
the implication $\eqref{pt1-v1206}\Rightarrow\eqref{pt2-v1206}$ is true. 

It remains to prove the moreover part. We separate two cases according to $k$.
\begin{itemize}
\item If $k=2$, then $\Rank A_{\widetilde \beta}=\Rank (\beta_0)=1$. So 1-atomic measure
exists if and only if $\Rank A_{(\beta_0,\beta_1,y_0)}=\Rank \widehat A(y_0)=1$ for some $y_0$. But from the form of $A_{(\beta_0,\beta_1,y)}$ and $\widehat A(y)$ this is possible only if 
$y_0=\frac{\beta_1^2}{\beta_0}= \sqrt{\beta_0\beta_4}$. Otherwise there is a 2-atomic measure.
\item 
Else $k>2$.
By Proposition \ref{rank-13-07} and \eqref{psd-cond-1-case2} above, 
	$\Rank A_{(\widetilde\beta,\beta_{2k-3},y_0)}=\Rank A_{\widetilde \beta}$ 
if and only if 
	$y_0=sA_{\widetilde{\beta}}^{+}s^T$. 
In the proof of the implication $\eqref{pt3-v1606}\Rightarrow\eqref{pt1-v1606}$ we see that
$\Rank A_{\widetilde  \beta}\geq \Rank \widetilde A$. 
Using this in \eqref{rank-cond-15-07-20} above, it follows that $sA_{\widetilde{\beta}}^{+}s^T$ must be
equal to $y_-$ or $y_+$, which is exactly \eqref{cond-to-solve-v1606-minimally}.
\end{itemize}
This concludes the proof of the theorem.
\end{proof}

The following corollary is a consequence of Theorem \ref{trunc-Hamb-without-3n-2-and-1} and solves the bivariate TMP for the curve $y=x^4$ where also
$\beta_{3,2k-2}$ is given.

\begin{corollary}\label{Y=X4-gen}
	Let 
		$\displaystyle\beta=(\beta_{i,j})_{i,j\in \ZZ^2_+,i+j\leq 2k}$  
	be a 2-dimensional real multisequence of degree $2k$ and let $\beta_{3,2k-2}$ be also given. 
	Suppose $M(k)$ is positive semidefinite and recursively generated.
	Let 
		$$u^{(i)}=(\beta_{0,i},\beta_{1,i},\beta_{2,i},\beta_{3,i})\quad \text{for }i=0,\ldots,2k-1,$$
	  $$\widehat{\beta}:=(u^{(0)},\ldots,u^{(2k-3)},\beta_{0,2k-2},\beta_{1,2k-2},\beta_{2,2k-2})
		\quad\text{and}\quad
		\widetilde{\beta}:=(\widehat{\beta},\beta_{3,2k-2},\beta_{0,2k-1})$$
	be subsequences of $\beta$,
	  $$u:=\left(\begin{array}{cccc}u^{(k)} & \cdots & u^{(2k-1)} &\beta_{1,2k-1}\end{array}\right),\quad
		s:=\left(\begin{array}{ccccc}\beta_{3,k-1} & u^{(k)} & \cdots & u^{(2k-1)} & \beta_{1,2k-1}
				\end{array}\right),$$
	  $$
		w:=\left(\begin{array}{cccccccc}\beta_{2,k-1} & \beta_{3,k-1} & u^{(k)}  & \cdots & 
				u^{(2k-2)} & \beta_{1,2k-2} & \beta_{2,2k-2} & \beta_{3,2k-2}\end{array}\right)
	  $$
	vectors and
		$$\widetilde{A}:=
		\left(\begin{array}{cc} 
			A_{\widehat\beta} & u^T \\ u & 
			\beta_{0,2k}\end{array}\right)$$
	a matrix.
	Then $\beta$ has a representing measure supported on $y=x^4$ if and only if 
		$$sA_{\widetilde{\beta}}^{+}s^T \leq u A_{\widehat{\beta}}^{+}w^T+\sqrt{(A_{\widetilde{\beta}}/
			A_{\widehat{\beta}}) (\widetilde{A}/A_{\widehat{\beta}})}.$$
	one of the following statements hold:
	\begin{enumerate}
	\item One of the following holds: 
		\begin{itemize}
			\item If $k\geq 4$, then $Y=X^4$ is a column relation of $M(k)$. 
			\item If $k=3$, then the equalities 
				$\beta_{0,1}=\beta_{4,0}$, $\beta_{1,1}=\beta_{5,0}$, $\beta_{2,1}=\beta_{6,0}$ hold.
			\item If $k=2$, then the equality $\beta_{0,1}=\beta_{4,0}$ holds.
			\item $k=1$.
		\end{itemize}
	\item One of the following conditions holds:
	\begin{enumerate}
		\item\label{yx4cond1}  $A_{\widetilde{\beta}}\succ 0$.
		\item $A_{\widetilde{\beta}}\succeq 0$ and $\Rank A_{\widehat{\beta}}=\Rank A_{\widetilde{\beta}}=
					\Rank \left(\begin{array}{cc}A_{\widetilde{\beta}}& s^T\end{array}\right)=\Rank \widetilde{A}$.
	\end{enumerate}
	\end{enumerate}
	Moreover, if the representing measure exists, then there is a $(\Rank \widetilde\beta)$-atomic measure if 
	\begin{equation*}
		sA_{\widetilde{\beta}}^{+}s^T \in 
			\left\{u A_{\widehat{\beta}}^{+}w^T- \sqrt{(A_{\widetilde{\beta}}/
			A_{\widehat{\beta}}) (\widetilde{A}/A_{\widehat{\beta}})},
			u A_{\widehat{\beta}}^{+}w^T+ \sqrt{(A_{\widetilde{\beta}}/
			A_{\widehat{\beta}}) (\widetilde{A}/A_{\widehat{\beta}})}\right\}.
	\end{equation*}
	and $(\Rank \widetilde\beta+1)$-atomic otherwise.
\end{corollary}

\begin{proof}
	For $m\in \{0,1\ldots,8k-3,8k\}$ we define the numbers $\widetilde \beta_m$ by the following rule
		\begin{equation*}
			\widetilde \beta_m:=\beta_{m\Mod{4},\lfloor\frac{m}{4}\rfloor}.
		\end{equation*}
	\noindent \textbf{Claim 1.} Every number $\widetilde \beta_m$ is well-defined.\\

	We will prove that $m\Mod{4}+\lfloor\frac{m}{4}\rfloor\leq 2k$ 
	if $m\neq 8k-5$, while for $m=8k-5$ we have $\widetilde\beta_{8k-5}=\beta_{3,2k-2}$.
	We separate three cases according to $m$.
	\begin{itemize}
		\item $m< 8k-8$:  $\lfloor\frac{m}{4}\rfloor+m\Mod{4}\leq (2k-3)+3=2k$.
		\item $m\in \{8k-8,8k-7,8k-6\}$:
			$\lfloor\frac{m}{4}\rfloor+m\Mod{4}\leq (2k-2)+2=2k$.
		\item $m\in \{8k-4,8k-3\}$:
			$\lfloor\frac{m}{4}\rfloor+m\Mod{4}\leq (2k-1)+1=2k$.
		\item $m=8k$: $\lfloor\frac{m}{4}\rfloor+m\Mod{3}= 2k+0=2k.$\\
	\end{itemize}
	
	\noindent \textbf{Claim 2.} 
		Let $t\in \NN$. 
		The atoms $(x_1,x_1^4),\ldots (x_t,x_t^4)$ with densities $\lambda_1,\ldots,\lambda_t$
		are the $(y-x^4)$-representing measure for $\beta$ and $\beta_{3,2k-2}$
		if and only if
		the atoms $x_1,\ldots,x_t$ with densities $\lambda_1,\ldots,\lambda_t$
		are the $\RR$-representing measure for 
		 $\widetilde \beta(x,y)=(\widetilde \beta_0,\ldots,\widetilde \beta_{2k-2},y,x,\widetilde\beta_{2k})$.\\

	The if part follows from the following calculation:
		$$\widetilde \beta_{m}
			=\beta_{m\Mod{4},\lfloor\frac{m}{4}\rfloor}
			=\sum_{\ell=1}^t \lambda_\ell x_\ell^{m\Mod{4}}x_{\ell}^{4\lfloor\frac{m}{4}\rfloor}
			=\sum_{\ell=1}^t \lambda_\ell x_\ell^{m\Mod{4}+4\lfloor\frac{m}{4}\rfloor}
			=\sum_{\ell=1}^t \lambda_\ell x_\ell^{m},$$
	where $m=0,\ldots,8k-3,8k$.

	The only if part follows from the following calculation for $i+j\leq 2k$:
	\begin{align*}
	  \beta_{i,j}
		&= \beta_{i-4,j+1}=\cdots=\beta_{i\Mod{4},j+\lfloor\frac{i}{4}\rfloor}\\
		&=\widetilde \beta_{i\Mod{4}+4(j+\lfloor\frac{i}{4}\rfloor)}
		  = \sum_{\ell=1}^t \lambda_\ell x_\ell^{i\Mod{4}+4(j+\lfloor\frac{i}{4}\rfloor)}
		  =\sum_{\ell=1}^t \lambda_\ell x_\ell^{i\Mod{4}+4\lfloor\frac{i}{4}\rfloor} x_{\ell}^{4j}
		  =\sum_{\ell=1}^t \lambda_\ell x_\ell^{i}(x_{\ell}^4)^{j},
	\end{align*}
	where the equalities in the first line follow by $M(k)$ being rg, and
	$$\beta_{3,2k-2}=\widetilde\beta_{8k-5}
		  =\sum_{\ell=1}^t \lambda_\ell x_\ell^{8k-5}
		  =\sum_{\ell=1}^t \lambda_\ell x_\ell^{3}(x_{\ell}^4)^{2k-2}.$$

	Using Claim 2 and a theorem of Bayer and Teichmann \cite{BT06}, implying that if a finite sequence 
	has a $K$-representing measure, then it has a finitely atomic $K$-representing measure, the statement of the 	
	Corollary follows by Theorem \ref{trunc-Hamb-without-3n-2-and-1}.
\end{proof}

\section{Truncated Hamburger moment problem of degree $2k$ with gap(s) $(\beta_{1})$, 
	$(\beta_{1},\beta_{2})$}
\label{S4}

In this section we solve the THMP of degree $2k$ with gaps $(\beta_{1})$ 
	(see Theorem \ref{trunc-Hamb-without-1})
and $(\beta_{1},\beta_{2})$ 
	(see Theorem \ref{trunc-Hamb-without-1-2}). 
As a corollary of Theorem \ref{trunc-Hamb-without-1} 
we obtain the solution to the TMP for the curve $y^2=x^3$ (see Corollary \ref{Y2=X3-general}),
while as a corollary of Theorem \ref{trunc-Hamb-without-1-2} 
we get the solution to the TMP for the curve $y^3=x^4$ and an additional moment $\beta_{\frac{5}{3},0}$ given
(see Corollary \ref{Y3=X4-gen}).

\subsection{Truncated Hamburger moment problem of degree $2k$ with gaps $(\beta_1)$}


\setlength{\columnsep}{-3cm}

\begin{theorem}\label{trunc-Hamb-without-1} 
 Let $k\in \NN$, $k>1$, and 
	$$\beta(x):=(\beta_0,x,\beta_2,\ldots,\beta_{2k})$$ 
  be a sequence where each $\beta_i$ is a real number, $\beta_0>0$ and $x$ is a variable. 
  Let 
	$$\widehat{\beta}:=(\beta_2,\ldots,\beta_{2k-2}),\quad
		\widetilde{\beta}:=(\beta_2,\ldots,\beta_{2k}),\quad
			\overline{\beta}:=(\beta_4,\ldots,\beta_{2k-2})\quad\text{and}\quad
				\widebreve{\beta}:=(\beta_4,\ldots,\beta_{2k})$$
  be subsequences of $\beta(x)$,
	$$v:=\left(\begin{array}{ccc}\beta_{2} & \cdots & \beta_{k-1}\end{array}\right)\quad\text{and}\quad
		u:=\left(\begin{array}{ccc}\beta_{2} & \ldots & \beta_{k}\end{array}\right)$$
  vectors, and
	$$\widetilde{A}:=\left(\begin{array}{cc} \beta_0 & v \\ v^T & A_{\overline{\beta}} \end{array}\right)
		\quad\text{and}\quad
			\widehat{A}:=\left(\begin{array}{cc} \beta_0 & u \\ u^T & A_{\widebreve{\beta}} \end{array}\right)$$
  matrices.
  Then the following statements are equivalent: 
\begin{enumerate}	
	\item\label{pt1-v1906} There exists $x_0\in \RR$ and a representing measure for $\beta(x_0)$ supported on $K=\RR$.
	\item\label{pt2-v1906} There exists $x_0\in \RR$ and a 
		$(\Rank \widetilde\beta)$ or 
		a $(\Rank \widetilde\beta+1)$-atomic representing measure for $\beta(x_0)$.
	\item\label{pt3-v1906} $A_{\beta(x)}$ is partially positive semidefinite and one of the following conditions is true: 
		\begin{enumerate}
			\item\label{pt3a-v1906} 
			  \begin{enumerate}
				\item $k=2$ and $A_{\widetilde{\beta}}\succ 0$.
				\item $k>2$, $A_{\widetilde{\beta}}\succ 0$ and $\widetilde{A}\succ 0$.
			  \end{enumerate}
			\item\label{pt3b-v1906} 
				$\Rank A_{\widehat\beta}=\Rank A_{\widetilde \beta}=\Rank A_{\breve \beta}.$
		\end{enumerate}
\end{enumerate}

Moreover, if the representing measure exists, then there does not exist a $(\Rank \widetilde\beta)$-atomic measure if and only if 
	\eqref{pt3b-v1906} holds and $\Rank A_{\widehat\beta}<\Rank \widehat A$.
\end{theorem}

\begin{proof}
First we prove the implication $\eqref{pt1-v1906}\Rightarrow\eqref{pt3-v1906}$.
By Theorem \ref{Hamburger}, $A_{\beta(x_0)}\succeq 0$ and $\Rank A_{\beta(x_0)}=\Rank\beta(x_0)$.
The condition $A_{\beta(x_0)}\succeq 0$ implies that $A_{\beta(x)}$ is ppsd. 
We separate two cases according to the invertibility of $A_{\widetilde\beta}$.
\begin{itemize}
  \item $A_{\widetilde\beta}\succ 0$: Since $A_{\widetilde\beta}$ is a principal submatrix of $A_{\beta(x_0)}$, we conclude that 
	$\Rank A_{\beta(x_0)}\geq \Rank A_{\widetilde \beta}=k$, and hence $A_{\beta(x_0)}$ is either invertible or 
	$\Rank A_{\beta(x_0)}$ is singular and by Corollary \ref{singular-case-measure} used for $\beta(x_0)$ as $\beta$, 
	$\Rank A_{\beta(x_0)}=\Rank A_{(\beta_0,x_0,\widehat\beta)}$.
	In both cases 
		$$A_{(\beta_0,x_0,\widehat\beta)}=
			\left\{\begin{array}{ll}
			\left(\begin{array}{cc}
				\beta_{0} & x_0\\
				x_0  & \beta_2 
			\end{array}\right),&\text{if }k=2,\\
			\left(\begin{array}{ccc}
				\beta_{0} & x_0 & v\\
				x_0  & \beta_2 & v_1\\
				v^T & v_1^T& A_{\overline{\beta}}
			\end{array}\right)\quad \text{where}\quad v_1=\left(\begin{array}{ccc}\beta_{3} & \cdots & \beta_{k}\end{array}\right),& \text{if }k>2,\end{array}\right.,$$ 
	is invertible. If $k>2$, $\widetilde A$ is a principal submatrix of $A_{(\beta_0,x_0,\widehat\beta)}$ and
	 it follows that $\widetilde A\succ 0$.
	Hence, \eqref{pt3a-v1906} holds. Together with $A_{\beta(x)}$ being ppsd, proves the implication $\eqref{pt1-v1906}\Rightarrow\eqref{pt3-v1906}$ in this case.
  \item $A_{\widetilde{\beta}}$ is singular: Since $\widetilde \beta$ is a subsequence of $\beta(x_0)$ of the form from Proposition \ref{measure-for-subsequence} with
	$i=1, j=k$, it admits a measure. By Corollary \ref{singular-case-measure} used for $\widetilde\beta$ as $\beta$, it follows that 
		\begin{equation}\label{rank-12:46-15-07-20}
			\Rank A_{\widetilde \beta}=\Rank A_{\widehat \beta}.
		\end{equation}
	By Corollary \ref{rank-theorem-3} used for $\beta(x_0)$ as $\beta$, it follows that 
		\begin{equation}\label{rank-12:48-15-07-20}
			\Rank A_{\widetilde \beta}=\Rank A_{\breve\beta}.
		\end{equation}		
	Hence, \eqref{rank-12:46-15-07-20} and \eqref{rank-12:48-15-07-20} imply that \eqref{pt3b-v1906} holds.
	Together with $A_{\beta(x)}$ being ppsd, proves the implication $\eqref{pt1-v1906}\Rightarrow\eqref{pt3-v1906}$ in this case.
\end{itemize}

Second we prove the implication $\eqref{pt3-v1906}\Rightarrow\eqref{pt2-v1906}$.
Let $P_1:\RR^{k+1}\to \RR^{k+1}$ be the following permutation matrix	
	$$P_1=
		\left(\begin{array}{ccc}
			\bf{0} & 0 & 1\\
			\bf{0} & 1 & 0\\
			I_{k-1} & 0 & 0
		\end{array}\right),$$
where $\bf{0}$ stands for the row of $k-2$ zeros and $I_{k-1}$ is the identity matrix of size $k-1$.
Then 
	$P_1^TA_{\beta(x)}P_1$ is of the form 
  $$
	P_1^TA_{\beta(x)}P_1=
	\left(\begin{array}{ccc}
			A_{\widebreve \beta}& w^T & u^T \\ 
			w & \beta_{2} & x \\ 
			u & x & \beta_0 \\\end{array}\right),$$
where $w=\left(\begin{array}{ccc} \beta_3 & \cdots & \beta_{k+1}\end{array}\right)$ is a vector.\\

\noindent \textbf{Claim.} $A_{\beta(x_0)}$ is psd if and only if 
	$$x_0\in 
		\left[
			uA_{\breve\beta}^+w^T-\sqrt{(A_{\widetilde \beta}/A_{\widebreve \beta})(\widehat A/A_{\widebreve\beta})},
			uA_{\breve\beta}^+w^T+\sqrt{(A_{\widetilde \beta}/A_{\widebreve \beta})(\widehat A/A_{\widebreve\beta})}
		\right]=:[x_-,x_+].$$
	Moreover,
	  \begin{equation}\label{rank-eq-10-07}
		\Rank A_{\beta(x_0)}:=
		  \left\{\begin{array}{rr}
			\max\{\Rank A_{\widetilde \beta},\Rank \widehat A\},& \text{if }x_0\in \{x_-,x_+\},\\
			\max\{\Rank A_{\widetilde \beta},\Rank \widehat A\}+1,& \text{if }x_0\in (x_-,x_+).
		  \end{array}\right.
	\end{equation}\\

Denoting the matrices 
	$$
	\cA:=\left(\begin{array}{cc}A_{\breve \beta} & w^T \\ w & \beta_{2}\end{array}\right)
		\quad\text{and}\quad
	\cB:=\left(\begin{array}{cc}A_{\breve \beta} & u^T \\ u & \beta_{0}\end{array}\right),
	$$ 
and the permutation matrix $P_2:\RR^k\to \RR^k$ by
	$$P_2=
		\left(\begin{array}{cc}
			\bf{0} &  1\\
			I_{k-1} & 0 
		\end{array}\right),$$
where $\bf{0}$ stands for the row of $k-1$ zeroes and $I_{k-1}$ the identity matrix of size $k-1$, we have that
	$$\cA=P_2^TA_{\widetilde \beta}P_2 \quad \text{and} \quad \cB=P_2^T \widehat A P_2.$$
In particular, 
	\begin{equation}\label{rank-14:04-15-07-20}
		\Rank \cA=\Rank A_{\widetilde \beta} \quad\text{and}\quad \Rank \cB=\Rank  \widehat A.
	\end{equation}
If \eqref{pt3a-v1906} holds, then $A_{\widetilde \beta}\succ 0$ implies that $A_{\breve \beta}\succ 0$.
If \eqref{pt3b-v1906} holds, then in particular $\Rank A_{\breve \beta}=\Rank A_{\widetilde \beta}=\Rank \cA$.
Hence, the assumption \eqref{assump-on-ranks} of Lemma \ref{psd-completion} used for $P_1^TA_{\beta(x)}P_1, A_{\breve \beta}, \cA,\cB$ as $A(x), A_1, A_2, A_3$, respectively,
is satisfied and using also  
	$\cA/A_{\breve\beta}=A_{\widetilde \beta}/A_{\breve\beta}$ and 
	$\cB/A_{\breve\beta}=\widehat A/A_{\breve\beta}$, Claim follows.\\

First assume that \eqref{pt3a-v1906} holds. We separate two cases according to the inverbility of $\widehat A$.
\begin{itemize}
	\item $\widehat{A}\succ 0$: 
		From $A_{\widetilde \beta}\succ 0$ and $\widehat{A}\succ 0$ it follows, using Proposition \ref{rank-13-07}, that 
		$A_{\widetilde \beta}/A_{\widebreve \beta}>0$ and $\widehat A/A_{\widebreve\beta}>0$.
		Hence by the definition of $x_{\pm}$, we have $x_-<x_+$ and by Claim, $A_{\beta(x_0)}\succ 0$ for $x_0\in (x_-,x_+)$.
		By Theorem \ref{Hamburger}, $(\Rank \beta(x_0))=(\Rank \widetilde\beta+1)$-atomic representing measure for $\beta(x_0)$ exists,
		which proves the implication $\eqref{pt3-v1906}\Rightarrow\eqref{pt2-v1906}$ in this case.
	\item $\widehat{A}$ is singular: From $A_{\widetilde \beta}\succ 0$ it follows that $A_{\breve \beta}\succ 0$. Since $\widehat A$ is singular, 
		 Proposition \ref{rank-13-07} implies that $\widehat A/A_{\widebreve\beta}=0$, and hence by the definition of $x_{\pm}$, we have $x_-=x_+$. 
		By Claim, $A_{\beta(x_{\pm})}\succeq 0$ with $\Rank A_{\beta(x_{\pm})}=\Rank A_{\widetilde \beta}$.
		We separate two cases according to $k$.
		\begin{itemize}
		\item $k=2$: Since $\widehat A=
		\left(\begin{array}{cc}
			\beta_0 & \beta_2\\
			\beta_2	& \beta_{4}	
		\end{array}\right)$ and $\beta_0>0$, it follows that the second (also the last) column of 
			$\widehat A$ is in the span of the first (also the others) one. 
		\item $k>2$: By assumptions $\widetilde A\succ 0$ and 
		$\widehat A=
		\left(\begin{array}{cc}
			\widetilde A & u_1^T\\
			u_1	& \beta_{2k}	
		\end{array}\right)$ 
		being singular, 
		where the $u_1$ is equal to $u_1=\left(\begin{array}{cccc} \beta_{k} & \beta_{k+2} & \cdots & \beta_{2k-1}\end{array}\right)$, 
		it follows that the last column of $\widehat A$ is in the span of the others.
		\end{itemize}
			By Lemma \ref{extension-principle}, the last column of $A_{\beta(x_{\pm})}$ is also in the span of the others
		and by Corollary \ref{singular-case-measure}, we have that $(\Rank \beta(x_\pm))=(\Rank \widetilde\beta)$-atomic representing measure for $\beta(x_\pm)$ exists,
		which proves the implication $\eqref{pt3-v1906}\Rightarrow\eqref{pt2-v1906}$ in this case.
\end{itemize}

Otherwise \eqref{pt3b-v1906} holds. Proposition \ref{rank-13-07} implies that 
	$A_{\widetilde \beta}/A_{\widehat\beta}=0$, and hence by the definition of $x_{\pm}$, we have $x_-=x_+$.
	By Claim, $A_{\beta(x_{\pm})}\succeq 0$.
	The assumption $\Rank A_{\widetilde \beta}=\Rank A_{\widehat \beta}$, also implies that 
	the last column of 
	$A_{\widetilde \beta}=
		\left(\begin{array}{cc}
			A_{\widehat \beta} & u_2^T\\
			u_2	& \beta_{2k}	
		\end{array}\right)$,
	where $u_2=\left(\begin{array}{ccc} \beta_{k} & \cdots & \beta_{2k-1}\end{array}\right)$,  
	is in the span of the others.
	By Lemma \ref{extension-principle}, the last column of $A_{\beta(x_{\pm})}$ is in the span of the others. 
Hence, by Corollary \ref{singular-case-measure},
$(\Rank \beta(x_{\pm}))$-atomic measure for $\beta(x_\pm)$ exists.
Since $\widetilde \beta$ is a subsequence of $\beta(x_0)$ of the form from Proposition \ref{measure-for-subsequence} with
	$i=1, j=k$, it admits a measure and hence Theorem \ref{Hamburger} implies that $\Rank A_{\widetilde \beta}=\Rank \widetilde \beta$. 
From \eqref{rank-eq-10-07}, it follows that:
\begin{itemize}
	\item If $\Rank \widehat A\leq \Rank A_{\widetilde\beta}$, then $\Rank \beta(x_{\pm})=\Rank A_{\widetilde\beta}=\Rank \widetilde\beta$.
	\item Else $\Rank \widehat A= \Rank A_{\widetilde\beta}+1$ and 
		$\Rank \beta(x_{\pm})=\Rank \widehat A=\Rank \widetilde\beta+1.$
\end{itemize}
This proves the implication $\eqref{pt3-v1906}\Rightarrow\eqref{pt2-v1906}$ in this case.

The implication $\eqref{pt2-v1906}\Rightarrow\eqref{pt1-v1906}$ is trivial.

It remains to prove the moreover part. Observe that in the proof of the implication 
$\eqref{pt3-v1906}\Rightarrow\eqref{pt2-v1906}$, $(\Rank \widetilde\beta)$-atomic measure might not
exist if \eqref{pt3a-v1906} holds with $\widehat A\succ 0$ and 
does not exist if \eqref{pt3b-v1906} holds with $\Rank A_{\widetilde \beta}<\Rank \widehat A$.
We will prove that in the first case there always exists a $(\Rank \widetilde\beta)$-atomic measure.
Assume that $A_{\widetilde \beta}\succ 0$ and $\widehat A\succ 0$.
We will prove that one of $A_{\beta(x_\pm)}$ or $A_{\beta(x_+)}$ satisfies
	\begin{equation}\label{21:44-15-07-20}
		\Rank A_{\beta(x_\pm)}=\Rank A_{\beta(x_\pm)}(k-1),
	\end{equation}
and hence by Corollary \ref{singular-case-measure},
a $(\Rank \beta(x_\pm))=(\Rank \widetilde\beta)$-atomic measure exists.
Using Proposition \ref{Schur3by3-ver2} with for $A_{\beta(x)}, A_{\widetilde \beta}, A_{\widehat \beta}$ as $K,N,C$, respectively, 
and denoting $u:=A_{\widetilde \beta}/A_{\widehat \beta}$, 
we have that
	$$
	f(x):=A_{\beta(x)}/A_{\widetilde \beta}
		=\left(\beta_0-e(x) A_{\widehat \beta}^{-1} e(x)^T\right)-
			\frac{1}{u}\left(\beta_{k}-e(x) A_{\widehat \beta}^{-1} z^T\right)^2
		=:g(x)-\frac{1}{u}h(x)^2,
	$$
where 
	$e(x):=\left(\begin{array}{cccc} x & \beta_2 & \cdots & \beta_{k-1}\end{array}\right)$ 
and 
	$z:=\left(\begin{array}{ccc} \beta_{k+1} & \cdots & \beta_{2k-1}\end{array}\right)$.
From the proof of the implication $\eqref{pt3-v1906}\Rightarrow\eqref{pt2-v1906}$, we know that $x_-<x_+$ 
and 
	\begin{equation}\label{zeros-21:39-15-07-20}
		f(x_-)=f(x_+)=0.	
\end{equation} 
Note that $g(x)=A_{(\beta_0,x,\widehat\beta)}/A_{\widehat \beta}$.
If 
	\begin{equation}\label{zeros-21:33-15-07-20}
		g(x_-)=g(x_+)=0,
	\end{equation} 
then $h(x_-)=h(x_+)=0$. But $h(x)$ is a linear function in $x$, so this is possible 
only if $h(x)=0$ for every $x\in \RR$.
This is possible only if
  \begin{equation}\label{singular-case}
		A_{\widehat \beta}^{-1} z^T=\left(\begin{array}{cccc} 0 & b_2 & \cdots & b_{k-1} \end{array}\right)^T\quad \text{for some}\;b_2,\ldots,b_k\in\RR 
		\qquad \text{and} \qquad
		\beta_k=\sum_{i=2}^{k-1} \beta_ib_i.
  \end{equation}
We write $(A_{\beta(x)})|_{S_1,S_2}$ for the restriction of $A_{\beta(x)}$ to rows from $S_1$ and columns from $S_2$.
Since $A_{\beta(x)}$ is a Hankel matrix, we have
	$$(A_{\beta(x)})_{\{1,\ldots,X^{k-1}\},\{X,\ldots,X^{k}\}}=(A_{\beta(x)})_{\{X,\ldots,X^k\},\{1,\ldots,X^{k-1}\}},$$
which is equal to 
		$$\left(\begin{array}{cc}
			 e(x) & \beta_{k}\\
			A_{\widehat \beta}  & z^T
			\end{array}\right)=
		\left(\begin{array}{cc}
			e(x)^T & A_{\widehat \beta}  \\
			\beta_{k}  & z
			\end{array}\right).$$
\eqref{singular-case} implies that the last column of $(A_{\beta(x)})_{\{1,\ldots,X^{k-1}\},\{X,\ldots,X^{k}\}}$ is in the span of the columns $2,\ldots,k-1$.  
From $(A_{\beta(x)})_{\{X,\ldots,X^k\},\{1,\ldots,X^{k-1}\}}$ this in particular implies that the last column of $A_{\widehat \beta}$
is in the span of the others and $A_{\widehat \beta}$ is singular, which is a contradiction with the assumption $A_{\widetilde \beta}\succ 0$.
Therefore \eqref{zeros-21:33-15-07-20} cannot be true and one of $g(x_-)$ and $g(x_+)$ is positive. 
By Proposition \ref{rank-13-07}, this means that $A_{(\beta_0,x_+,\widehat\beta)}\succ 0$ or $A_{(\beta_0,x_-,\widehat\beta)}\succ 0$
and hence $\Rank A_{(\beta_0,x_-,\widehat\beta)}=k$ or $\Rank A_{(\beta_0,x_+,\widehat\beta)}=k$.
By Proposition \ref{rank-13-07} and \eqref{zeros-21:39-15-07-20}, $\Rank A_{\beta(x_-)}=\Rank A_{\beta(x_+)}=\Rank A_{\widetilde \beta}=k$.
Therefore $\Rank A_{\beta(x_-)}=\Rank A_{(\beta_0,x_-,\widehat\beta)}$ or $\Rank A_{\beta(x_+)}=\Rank A_{(\beta_0,x_+,\widehat\beta)}$.
Noticing that $A_{(\beta_0,x_\pm,\widehat\beta)}=A_{\beta(x_\pm)}(k-1)$, it follows that one of $x_\pm$ satisfies \eqref{21:44-15-07-20}.
This concludes the proof of the moreover part.
\end{proof}

\begin{remark}
	For $k=1$, the THMP with gaps $(\beta_1)$ coincides with the THMP with gaps $(\beta_{2k-1})$
	and hence the case $k=1$ is already covered by  Theorem \ref{trunc-Hamb-without-3n-1-bg}.
\end{remark}

\begin{example}\label{ex-v2}
	For $k=9$, let
		\begin{align*}
			\beta^{(1)}(x) &=\Big(1,x,11,0,\frac{979}{5},0,4103,0,\frac{462979}{5},0,2174855,
						0,\frac{261453379}{5},0,1275350087,0,\frac{156925970179}{5},\\
			&\hspace{0,5cm} 0,776760884999\Big),\\
			\beta^{(2)}(x) &=\Big(1,x,\frac{15}{2}, 0, \frac{177}{2}, 0, \frac{2445}{2}, 0, \frac{36177}{2},
			0,\frac{554325}{2},0,\frac{8656377}{2},0,\frac{136617405}{2},0,\frac{2169039777}{2},0,\\
			&\hspace{0,5cm} \frac{138214318741}{8}\Big),\\	
			\beta^{(3)}(x) &=\Big(1,x,\frac{15}{2}, 0, \frac{177}{2}, 0, \frac{2445}{2}, 0, \frac{36177}{2},
			0,\frac{554325}{2},0,\frac{8656377}{2},0,\frac{136617405}{2},0,\frac{2169039777}{2},0,\\
			&\hspace{0,5cm} \frac{34553579685}{2}\Big),\\
			\beta^{(4)}(x) &=\frac{1}{9}(9,x,133, -235, 3157, -7987, 86893, -281995, 2598757,-10096867,82154653,-362972155,\\ 		
			&\hspace{0,5cm} 2699153557,-13062280147,91112865613,-470199300715,3134918735557,-16926788453827,\\
			&\hspace{0,5cm} 109327177835773),
		\end{align*}
	Let $\widetilde A^{(i)}$ and $\widehat A^{(i)}$, $i=1,2,3$, denote $\widetilde A$,  $\widehat A$, respectively,
	from Theorem \ref{trunc-Hamb-without-1}	corresponding to $\beta^{(i)}(x)$. 
	Using \textit{Mathematica} \cite{Wol} one can check that:
	\begin{itemize}
		\item  $\widehat A^{(1)}\succ 0$, while for $i=2,3,4$ it holds that
			$\widehat A^{(i)}\succeq 0$ and $\dim\left(\ker \widehat A^{(i)}\right)=1$.
		\item For $i=1,4$ we have $\widetilde A^{(i)}\succ 0$ for $i=1,4$, while for $i=2,3$ it holds that
				$\widetilde A^{(i)}\succeq 0$ and $\dim\left(\ker \widetilde A^{(i)}\right)=1$.
		\item $A_{\widetilde \beta^{(i)}}\succ 0$ for $i=1,2,4$,
				$A_{\widetilde \beta^{(3)}}\succeq 0$ and $\dim\left(\ker A_{\widetilde \beta^{(3)}}\right)=1$.
		\item $A_{\widehat \beta^{(3)}}\succ 0$ and $A_{\breve \beta^{(3)}}\succ 0$.
	\end{itemize}
	Therefore:
	\begin{itemize}
		\item  $A_{\beta^{(1)}(x)}$ is ppsd and \eqref{pt3a-v1906} of 		
			Theorem \ref{trunc-Hamb-without-1} is true, implying that a 9-atomic measure for 
			$\beta^{(1)}(x)$ exists.
		\item $\beta^{(2)}(x)$ does not satisfy 
			$\eqref{pt3a-v1906}$ neither \eqref{pt3b-v1906} of Theorem \ref{trunc-Hamb-without-1},
			implying there is no representing measure for $\beta^{(2)}(x)$.
		\item $A_{\beta^{(3)}(x)}$ is ppsd and $\widetilde \beta^{(3)}$ satisfies \eqref{pt3b-v1906} of 		
			Theorem \ref{trunc-Hamb-without-1} together with
			$\Rank A_{\widehat \beta^{(3)}}=\Rank \widehat A^{(3)}$, implying that an 8-atomic measure for 
			$\beta^{(3)}(x)$ exists.
		\item $A_{\beta^{(4)}(x)}$ is ppsd and $\widetilde \beta^{(4)}$ satisfies \eqref{pt3a-v1906} of 		
			Theorem \ref{trunc-Hamb-without-1}, implying that a 9-atomic measure for 
			$\beta^{(4)}(x)$ exists.
	\end{itemize}
\end{example}

The following corollary is a consequence of Theorem \ref{trunc-Hamb-without-1} and gives the solution of the bivariate TMP for the curve $y^2=x^3$. 

\begin{corollary}
	\label{Y2=X3-general}
	Let $\displaystyle\beta=(\beta_{i,j})_{i,j\in \ZZ^2_+,i+j\leq 2k}$ 
	be a 2-dimensional real multisequence of degree $2k$. 
	Suppose $M(k)$ is positive semidefinite and recursively generated.
	Let 
	$$u^{(i)}:=(\beta_{1,i},\beta_{0,i+1},\beta_{2,i}) \quad\text{for }i=0,\ldots,2k-2,$$
	$$\widehat{\beta}:=(u^{(0)},\ldots,u^{(2k-2)}),\quad
		\widetilde{\beta}:=(\widehat \beta,\beta_{1,2k-1},\beta_{0,2k}),\quad
		\overline{\beta}:=(\beta_{2,0},u^{(1)},\ldots,u^{(2k-2)})$$
	$$\text{and}\quad
		\widebreve{\beta}:=(\overline \beta,\beta_{1,2k-1},\beta_{0,2k})$$
	be subsequences of $\beta$,
		$$v:=\left(\begin{array}{cccc} u^{(0)} & \cdots & u^{(k-2)} & \beta_{1,k-1}\end{array}\right)$$
	a vector 
	and
		$$\widetilde{A}:=\left(\begin{array}{cc} \beta_0 & v \\ v^T & A_{\overline{\beta}} \end{array}\right)$$
	a matrix.
	Then $\beta$ has a representing measure supported on $y^2=x^3$ if and only if the following 
	statements hold:
	\begin{enumerate}
		\item One of the following holds: 
			\begin{itemize}
				\item If $k\geq 3$, then $Y^2=X^3$ is a column relation of $M(k)$. 
				\item If $k=2$, then the equalities 
					$\beta_{0,2}=\beta_{3,0}$, $\beta_{1,2}=\beta_{4,0}$, $\beta_{0,3}=\beta_{3,1}$ hold.
				\item $k=1$.
			\end{itemize}
		\item One of the following holds:
		\begin{enumerate}
			\item\label{y2x3-cond1} $A_{\widetilde\beta}\succ 0$ 
				and $\widetilde{A}\succ 0.$
			\item\label{y2x3-cond1-210720} $A_{\widetilde\beta}\succeq 0$ and 
				$\Rank A_{\widehat\beta}=\Rank A_{\widetilde\beta}=\Rank A_{\breve\beta}.$
		 \end{enumerate}
	\end{enumerate}
	Moreover, if the representing measure exists, then 
	there exists a  $(\Rank \widetilde\beta)$-atomic measure if 
	\eqref{y2x3-cond1} is true or \eqref{y2x3-cond1-210720} holds with
	$\Rank A_{\widehat\beta}=\Rank M(k)$. 
	Otherwise there is a $(\Rank \widetilde\beta+1)$-atomic measure.
\end{corollary}

\begin{proof}

	For $m\in \{0,2\ldots,6k\}$ we define the numbers $\widetilde \beta_m$ by the following rule
		\begin{equation*}
			\widetilde \beta_m:=
			\left\{\begin{array}{rl}
				\beta_{0,\frac{m}{3}},& \text{if }m\Mod{3}=0,\\
				\beta_{2,\lfloor\frac{m}{3}\rfloor-1},& \text{if }m\Mod{3}=1,\\
				\beta_{1,\lfloor\frac{m}{3}\rfloor},& \text{if }m\Mod{3}=2.
			\end{array}\right.
		\end{equation*}
	\noindent \textbf{Claim 1.} Every number $\widetilde \beta_m$ is well-defined.\\

	We have to prove that $i+j\leq 2k$, where $i,j$ are indices of $\beta_{i,j}$ used in the definition of 
	$\widetilde \beta_m$.
	We separate three cases according to $m$:
	\begin{itemize}
		\item $m\Mod{3}=0$: $\frac{m}{3}\leq 2k$.
		\item $m\Mod{3}=1$: $2+(\lfloor\frac{m}{3}\rfloor-1)\leq 2+(2k-2)=2k.$
		\item $m\Mod{3}=2$: $1+\lfloor\frac{m}{3}\rfloor\leq 1+(2k-1)=2k.$\\
	\end{itemize}
	
	\noindent \textbf{Claim 2.} 
		Let $t\in \NN$. 
		The atoms $(x_1^2,x_1^3),\ldots (x_t^2,x_t^3)$ with densities $\lambda_1,\ldots,\lambda_t$
		are the $(y^2-x^3)$-representing measure for $\beta$
		if and only if
		the atoms $x_1,\ldots,x_t$ with densities $\lambda_1,\ldots,\lambda_t$
		are the $\RR$-representing measure for 
		 $\widetilde \beta(x)=(\widetilde \beta_0,x,\widetilde \beta_2,\ldots,\widetilde\beta_{2k})$.\\

	The if part follows from the following calculation:
		$$\widetilde \beta_{m}
			=\left\{\begin{array}{rl}
				\beta_{0,\frac{m}{3}},& \text{if }m\Mod{3}=0,\\
				\beta_{2,\lfloor\frac{m}{3}\rfloor-1},& \text{if }m\Mod{3}=1,\\
				\beta_{1,\lfloor\frac{m}{3}\rfloor},& \text{if }m\Mod{3}=2,
				\end{array}\right.
			=\left\{\begin{array}{rl}
				\sum_{\ell=1}^t \lambda_\ell (x_{\ell}^3)^{\frac{m}{3}},& \text{if }m\Mod{3}=0,\\
				\sum_{\ell=1}^t \lambda_\ell (x_{\ell}^2)^2(x_{\ell}^3)^{\lfloor\frac{m}{3}\rfloor-1},& 
					\text{if }m\Mod{3}=1,\\
				\sum_{\ell=1}^t \lambda_\ell x_{\ell}^2(x_{\ell}^3)^{\lfloor\frac{m}{3}\rfloor},& 
					\text{if }m\Mod{3}=2,\\
				\end{array}\right.
			=\sum_{\ell=1}^t \lambda_\ell x_\ell^{m},$$
	where $m=0,2,\ldots,6k$.

	The only if part follows from the following calculation:
	\begin{align*}
	  \beta_{i,j}
		&= \beta_{i-3,j+2}=\cdots=\beta_{i\Mod{3},j+2\lfloor\frac{i}{3}\rfloor}
		  =\widetilde \beta_{2(i\Mod{3})+3(j+2\lfloor\frac{i}{3}\rfloor)}\\
		&= \sum_{\ell=1}^t \lambda_\ell x_\ell^{2(i\Mod{3})+3(j+2\lfloor\frac{i}{3}\rfloor)}
		  =\sum_{\ell=1}^t \lambda_\ell x_\ell^{2(i\Mod{3}+3\lfloor\frac{i}{3}\rfloor)} x_{\ell}^{3j}
		  =\sum_{\ell=1}^t \lambda_\ell (x_\ell^2)^{i}(x_{\ell}^3)^{j},
	\end{align*}
	where the first three equalities in the first line follow by $M(k)$ being rg.\\

	Using Claim 2 and a theorem of Bayer and Teichmann \cite{BT06}, implying that if a finite sequence 
	has a $K$-representing measure, then it has a finitely atomic $K$-representing measure, the statement of the 	
	Corollary follows by Theorem \ref{trunc-Hamb-without-1}.
\end{proof}

\subsection{Truncated Hamburger moment problem of degree $2k$ with gaps $(\beta_{1},\beta_{2})$}


\setlength{\columnsep}{1cm}

\begin{theorem}\label{trunc-Hamb-without-1-2} 
  Let $k\in \NN$, $k>2$, and 
	$$\beta(x,y):=(\beta_0,x,y,\beta_3,\ldots,\beta_{2k})$$
  be a sequence, where each $\beta_i$ is a real number, $\beta_0>0$ and $x,y$ are variables.   
  Let 
	$$\widetilde{\beta}:=(\beta_4,\ldots,\beta_{2k-2}),\;\;
		\overline\beta:=(\beta_{4},\ldots,\beta_{2k}),\;\;
		\breve\beta:=(\beta_6,\ldots,\beta_{2k-2}),\;\;
		\overline{\overline{\beta}}:=(\beta_{6},\ldots,\beta_{2k})$$
be subseqeunces of $\beta(x,y)$,
	$$v:=\left(\begin{array}{ccc}\beta_{3} & \dots & \beta_{k-1}\end{array}\right),\;\;
			u:=\left(\begin{array}{ccc}\beta_{3} & \cdots & \beta_{k}\end{array}\right),\;\;
				s:=\left(\begin{array}{ccc}\beta_{3} & \cdots & \beta_{k+1}\end{array}\right),$$
					$$w:=\left(\begin{array}{ccc}\beta_{5} & \cdots &\beta_{k+2}\end{array}\right),$$
vectors,
and
	$$\overline{A}:=
		\left(\begin{array}{cc} \beta_0 & v \\ v^T & A_{\breve \beta}\end{array}\right)
		\quad\text{and}\quad
		\widetilde{A}:=\left(\begin{array}{cc} \beta_0 & u \\ u^T & A_{\overline{\overline{ \beta}}}\end{array}\right)$$
matrices.
  Then the following statements are equivalent: 
\begin{enumerate}	
	\item\label{pt1-v2306} There exist $x_0,y_0\in \RR$ and a representing measure for $\beta(x_0,y_0)$ supported on $K=\RR$.
	\item\label{pt2-v2306} There exist $x_0,y_0\in \RR$ and a 
		$(\Rank \overline\beta)$ or 
		$(\Rank \overline\beta+1)$-atomic representing measure for $\beta(x_0,y_0)$.
	\item\label{pt3-v2306} $A_{\beta(x,y)}$ is partially positive semidefinite,
	\begin{equation}\label{cond-to-solve-v2306}
		 	sA_{\overline{\beta}}^{+}s^T 
			\leq u A_{\overline{\overline\beta}}^+w^T+\sqrt{(A_{\overline{\beta}}/A_{\overline{\overline{\beta}}}) (\widetilde A/A_{\overline{\overline \beta}})}
	\end{equation}
	and one of the following statements is true: 
	\begin{enumerate}
	\item\label{case1-2306} $A_{\overline{\beta}}\succ 0$
		and one of the following holds:
		\begin{enumerate}
			\item\label{subcase1-2606} 
			  \begin{enumerate}
				\item $k=3$ and the inequality in \eqref{cond-to-solve-v2306} is strict..
				\item $k>3$, $\overline A\succ 0$ and the inequality in \eqref{cond-to-solve-v2306} is strict.
			  \end{enumerate}
			\item\label{subcase2-2606} The following inequalities holds:
				$$uA_{\widetilde{\beta}}^{+}u^T < sA_{\overline{\beta}}^{+}s^T
					\quad\text{and}\quad 
					u A_{\overline{\overline\beta}}^+w^T-
					\sqrt{(A_{\overline{\beta}}/A_{\overline{\overline{\beta}}}) 
					(\widetilde A/A_{\overline{\overline \beta}})}
					\leq sA_{\overline{\beta}}^{+}s^T.$$
		\end{enumerate}
	\item\label{case3-2306} $\Rank A_{\widetilde \beta}=\Rank A_{\overline \beta}=
											\Rank \left(\begin{array}{cc} s^T & A_{\overline \beta}\end{array}\right)$.
	\end{enumerate}
\end{enumerate}

Moreover, if the representing measure exists, then there is a $(\Rank \overline\beta)$-atomic if
and only if \eqref{subcase2-2606} or \eqref{case3-2306} holds.
\end{theorem}

\begin{proof}
Note that $\beta(x,y)$ admits a measure if and only if there exist $y_0\in \RR$
such that $\beta(x,y_0)$ admits a measure.
Theorem \ref{trunc-Hamb-without-1} implies the following claim holds.\\

\noindent \textbf{Claim 1.} $\beta(x,y_0)$ admits a measure if and only if 
the following conditions hold:
\begin{enumerate}
	\item $A_{\beta(x,y_0)}$ is ppsd.
	\item Denoting 
		$$\widetilde\beta(y_0):=(y_0,\beta_3,\ldots,\beta_{2k-2})\quad\text{and}\quad
			\overline\beta(y_0):=(y_0,\beta_3,\ldots,\beta_{2k}),$$ 
	one of the following is true:
	\begin{enumerate}
		\item\label{cond1-proof-2306}  $A_{\overline\beta(y_0)}=
				\left(\begin{array}{cc} y_0 & s \\ s^T 
									& A_{\overline{\beta}} \end{array}\right)
				\succ 0$ 
			and 
				$\overline A(y_0)\succ 0$,
			where
			$$\overline A(y)
				:=\left\{\begin{array}{ll}
					\left(\begin{array}{c|c}\beta_0 & \begin{array}{cc} y & \beta_3 \end{array}\\\hline
						\begin{array}{c} y\\\beta_3 \end{array} & A_{\widetilde{\beta}} \end{array}\right),&
					\text{if }k=3,\\
					\left(\begin{array}{cc} \beta_0 & v(y) \\ v(y)^T 
									& A_{\widetilde{\beta}} \end{array}\right)
				=\left(\begin{array}{ccc} 
							\beta_0 & y & v \\ 
							y & \beta_{4} & w_1\\
							v^T & w_1^T & A_{\breve{\beta}} \end{array}\right),
			\quad	v(y)^T=\left(\begin{array}{c} y \\ v\end{array}\right)
				\;\;\text{and}\;\; 
			w_1^T=	
				\left(\begin{array}{ccc} 
					\beta_{5} \\ \vdots \\ \beta_{k+2}
				\end{array}\right),& \text{otherwise}.\end{array}\right.$$
		\item\label{cond2-proof-2306} 
			$\Rank A_{\widetilde\beta(y_0)}=\Rank A_{\overline\beta(y_0)}=\Rank A_{\overline\beta}.$\\
	\end{enumerate}
\end{enumerate}

We denote by
	$$\widetilde A(y):=\left(\begin{array}{cc} \beta_0 & u(y) \\ u(y)^T 
									& A_{\overline{\beta}} \end{array}\right)\quad \text{where} \quad u(y)=\left(\begin{array}{cc} y & u\end{array}\right).$$

\noindent\textbf{Claim 2.}
Assume $A_{\overline{\overline\beta}}\succ 0$ or $\Rank A_{\overline{\overline\beta}}=\Rank A_{\overline \beta}.$ 
Then $\widetilde A(y_0)\succeq 0$ if and only if
\begin{equation}\label{psd-cond-1-case3-2306}
	\widetilde A(y)\;\text{is ppsd}\quad \text{and}\quad
		y_0\in 
		\Big[u A_{\overline{\overline\beta}}^+w^T-\sqrt{(A_{\overline{\beta}}/A_{\overline{\overline{\beta}}}) (\widetilde A/A_{\overline{\overline \beta}})},
			u A_{\overline{\overline\beta}}^+w^T+\sqrt{(A_{\overline{\beta}}/A_{\overline{\overline{\beta}}})(\widetilde A/A_{\overline{\overline\beta}})}\Big]=:[y_-,y_+].
\end{equation}
Moreover,
  \begin{equation}\label{rank-cond-11-07}
	\Rank \widetilde A(y_0)=
	\left\{\begin{array}{rl} 
		\max\{\Rank A_{\overline \beta}, \Rank \widetilde A\}, & y\in \{y_-,y_+\},\\
		\max\{\Rank A_{\overline \beta}, \Rank \widetilde  A\}+1, & y\in (y_-,y_+).
	 \end{array}\right.
	\end{equation}\\
\indent
Let $P_2$ be the permutation matrix as in the proof of Theorem \ref{trunc-Hamb-without-1}.
We have that
	$P_2^T\widetilde A(y)P_2$ is of the form 
\begin{equation}\label{permutation-2306}
	P_2^T\widetilde A(y) P_2=
	\left(\begin{array}{ccc}
			A_{\overline{\overline \beta}}& w^T & u^T \\ 
			w & \beta_{4} & y \\ 
			u & y & \beta_0 \\\end{array}\right),
\end{equation}
and denoting the matrices 
	$$
	\cA:=\left(\begin{array}{cc}A_{\overline{\overline \beta}} & w^T \\ w & \beta_{4}\end{array}\right)
		\quad\text{and}\quad
	\cB:=\left(\begin{array}{cc}A_{\overline{\overline \beta}} & u^T \\ u & \beta_{0}\end{array}\right),
	$$
and the permuation matrix $P_3:\RR^{k-1}\to \RR^{k-1}$ by
	$$P_3=
		\left(\begin{array}{cc}
			\bf{0} &  1\\
			I_{k-2} & 0
		\end{array}\right),$$
where $\bf{0}$ stands for the row of $k-2$ zeros and $I_{k-2}$ is the identity matrix of size $k-2$,
we have that
	\begin{equation}\label{rank-9;49-17-07-20}
		\cA=P_3^TA_{\overline \beta}P_3 \quad \text{and} \quad \cB=P_3^T \widetilde A P_3.
	\end{equation}
By the assumptions in Claim 2 and \eqref{rank-9;49-17-07-20}, $A_{\overline{\overline \beta}}\succ 0$ or 
$\Rank A_{\overline{\overline \beta}}=\Rank \cA$.
Hence, the assumption \eqref{assump-on-ranks} of Lemma \ref{psd-completion} used for 
$P_2^T\widetilde A(y) P_2, A_{\overline{\overline\beta}}, \cA,\cB$ as $A(x), A_1, A_2, A_3$, respectively,
is satisfied and using also 
	$\cA/A_{\overline{\overline\beta}}=A_{\overline\beta}/A_{\overline{\overline\beta}},$
	$\cB/A_{\overline{\overline\beta}}=\widetilde A/A_{\overline{\overline\beta}}$, Claim 2 follows.\\

Theorem \ref{block-psd} implies the following claim.\\

\noindent\textbf{Claim 3.} 
It is true that:
\begin{enumerate}
	\item\label{pt1-14:59-16-07} $A_{\overline\beta(y_0)}\succeq 0\quad$ if and only if
\begin{equation}\label{psd-cond-1-case2-2306v1}
	A_{\overline\beta}\succeq 0, \qquad 
	s^T\in\cC(A_{\overline\beta})
	\qquad\text{and}\qquad
	A_{\overline\beta(y_0)}/A_{\overline\beta}=y_0-sA_{\overline\beta}^{+}s^T\geq 0.
\end{equation} 
	\item $A_{\widetilde\beta(y_0)}\succeq 0\quad$ if and only if
\begin{equation}\label{psd-cond-1-case2-2306v1-16-07}
	A_{\widetilde\beta}\succeq 0, \qquad 
	u^T\in\cC(A_{\widetilde\beta})
	\qquad\text{and}\qquad
	A_{\widetilde\beta(y_0)}/A_{\widetilde\beta}=y_0-uA_{\widetilde\beta}^{+}u^T\geq 0.
\end{equation} 
\end{enumerate}

\noindent\textbf{Claim 4.} 
	Assume $A_{\overline{\overline\beta}}\succ 0$ or $\Rank A_{\overline{\overline\beta}}=\Rank A_{\overline \beta}.$ 
	Then $A_{\beta(x,y_0)}$ is ppsd for some $y_0\in \RR$ 
		if and only if
	$A_{\beta(x,y)}$ is ppsd, $s^T\in\cC(A_{\overline\beta})$  and  \eqref{cond-to-solve-v2306} holds.\\

Note that 
	$A_{\beta(x,y_0)}$  is ppsd 
		if and only if 
	$A_{\overline\beta(y_0)}\succeq 0$ and $\widetilde{A}(y_0)\succeq 0$.
The first condition of \eqref{psd-cond-1-case3-2306} (which also includes the first condition of \eqref{psd-cond-1-case2-2306v1})
is equivalent to $A_{\beta(x,y)}$ being ppsd. Further on, $y_0$ satisfying the third condition of \eqref{psd-cond-1-case2-2306v1}
and the second condition of \eqref{psd-cond-1-case3-2306} exists if and only if \eqref{cond-to-solve-v2306}  holds. 
This proves Claim 4.\\

First we prove the implication $\eqref{pt1-v2306}\Rightarrow\eqref{pt3-v2306}$.
By Claim 1, in particular $A_{\beta(x,y_0)}$ (and hence also $A_{\beta(x,y)}$)  is ppsd.
Since $\widetilde\beta(y_0)$ also admits a measure by 
Proposition \ref{measure-for-subsequence},
we either have $A_{\widetilde\beta(y_0)}\succ 0$ and in particular 
$A_{\overline{\overline\beta}}\succ 0$, or $A_{\widetilde\beta(y_0)}$ is singular and
it follows by  Corollary \ref{rank-theorem-3} that 
$A_{\overline{\overline\beta}}\succ 0$ or $\Rank A_{\overline{\overline\beta}}=\Rank A_{\overline \beta}.$ 

If \eqref{cond1-proof-2306} of Claim 1 holds, then in particular $A_{\overline \beta}\succ 0$
and if $k>3$ also $\overline A\succ 0$. Since $A_{\overline \beta}(y_0)\succ 0$, it follows using Proposition \ref{rank-13-07}
that $A_{\overline \beta}(y_0)/A_{\overline \beta}>0$ or equivalently
$y_0>sA_{\overline \beta^+}s^T$. Since by Claim 2, $y_0\in [y_-,y_+]$, this implies that
$sA_{\overline \beta^+}s^T<y_+$ which means that the inequality in 
\eqref{cond-to-solve-v2306} is strict.
Hence, $A_{\beta(x,y})$ is ppsd, $A_{\overline \beta}\succ 0$ and \eqref{subcase1-2606}
holds. This proves the implication  $\eqref{pt1-v2306}\Rightarrow\eqref{pt3-v2306}$ in this case.

Assume now that \eqref{cond2-proof-2306} of Claim 1 holds. 
There are two cases to consider:
\begin{itemize}
  \item $A_{\overline\beta}\succ 0$: It follows that 
	$\Rank A_{\widetilde{\beta}(y_0)}=\Rank A_{\overline{\beta}(y_0)}=k-1$, which implies that:
	\begin{itemize}
	  \item $A_{\widetilde{\beta}(y_0)}\succ 0$ since $A_{\widetilde{\beta}(y_0)}$ is of size $k-1$.
	  \item	By Proposition \ref{rank-13-07},	$y_0=sA_{\overline\beta}^{+}s^T$ since 
			$A_{\overline{\beta}(y_0)}=
				\left(\begin{array}{cc}
					y_0 & s\\ 
					s^T& A_{\overline\beta}
				\end{array}\right)$
			is singular.
	  \item	$k-1\leq \Rank A_{\beta(x_0,y_0)}\leq k$ for some $x_0\in \RR$ such that 
		$A_{\beta(x_0,y_0)}\succeq 0$,
		since 
		$$A_{\beta(x_0,y_0)}=
				\left(\begin{array}{cc}
					\beta_0 & u(x_0,y_0)\\ 
					u(x_0,y_0)^T & A_{\overline\beta(y_0)}
				\end{array}\right)\quad\text{where}\quad
					u(x_0,y_0)=\left(\begin{array}{ccc} x_0 & y_0 & u \end{array}\right).$$
	\end{itemize}
	From $A_{\widetilde{\beta}(y_0)}\succ 0$ and $y_0=sA_{\overline\beta}^{+}s^T$,
	it follows by Proposition \ref{rank-13-07}
	and \eqref{psd-cond-1-case2-2306v1-16-07} that
	$uA_{\widetilde{\beta}}^{+}u^T<sA_{\overline\beta}^{+}s^T$.
	Further on, 
	$y_-\leq sA_{\overline\beta}^{+}s^T$ since by Claim 2, 
	$\widetilde A(y_0)\succeq 0$ implies that $y_0\in [y_-,y_+]$.
	Hence, $A_{\beta(x,y})$ is ppsd, $A_{\overline \beta}\succ 0$ and \eqref{subcase2-2606}
	holds. 
	This proves the implication  $\eqref{pt1-v2306}\Rightarrow\eqref{pt3-v2306}$ in this case.
  \item $A_{\overline\beta}\not\succ 0$: By Lemma \ref{measure-for-subsequence}, 
	$\overline\beta$ also  admits a measure
	and hence by Corollary \ref{singular-case-measure} used for 
	$\overline{\beta}$ as $\beta$, $\Rank A_{\widetilde \beta}=\Rank A_{\overline\beta}$.
	Together with the second condition in \eqref{psd-cond-1-case2-2306v1}, 	
	this implies that \eqref{case3-2306} holds.
	Since $A_{\beta(x,y)}$ is ppsd and  \eqref{cond-to-solve-v2306} holds,
	this proves the implication  $\eqref{pt1-v2306}\Rightarrow\eqref{pt3-v2306}$ in this case.\\
\end{itemize}

Second we prove the implication $\eqref{pt3-v2306}\Rightarrow\eqref{pt1-v2306}$.
If \eqref{case1-2306} holds, then $A_{\overline\beta}\succ 0$ and in particular
$A_{\overline{\overline\beta}}\succ 0$. 
Else \eqref{case3-2306} holds and in particular $A_{\overline\beta}$ is singular.
By Claim 3, $A_{\overline\beta(sA_{\overline \beta}^+s^T)}\succeq 0$ and hence 
by Corollary \ref{rank-theorem-3} used for $\overline\beta(sA_{\overline \beta}^+s^T)$
as $\beta$ we conclude that $\Rank A_{\overline{\overline\beta}}=\Rank A_{\overline\beta}.$
Hence the assumption of Claims 2 and 4 is satisfied and
$A_{\beta(x,y_0)}$ is ppsd for every $y_0$ from the interval 
	$[\max\{y_-,sA_{\overline \beta}^+s^T\},y_+]$.
We separate cases three cases according to the assumptions:

\begin{itemize}
  \item Case \eqref{subcase1-2606}: We separate two cases according to the invertibility 
	of $\widetilde{A}$.
	\begin{itemize}
  		\item $\widetilde{A}\succ 0$: 
			Since $A_{\overline{\beta}}\succ 0$ and $\widetilde{A}\succ 0$, it follows that
			$A_{\overline{\beta}}/A_{\overline{\overline \beta}}>0$ and 
			$\widetilde A/A_{\overline{\overline \beta}}>0$. 
			By the form of $y_\pm$ given in Claim 2, we have that $y_-<y_+$. 
			Since by assumption also the inequality \eqref{cond-to-solve-v2306} is strict,
			the interval $(\max\{y_-,sA_{\overline \beta}^+s^T\},y_+)$ is not empty and 
			hence for every $y_0\in  (\max\{y_-,sA_{\overline \beta}^+s^T\},y_+)$, 		
			$A_{\beta(x,y_0)}$ satisfies \eqref{cond1-proof-2306} above by Claims 2 and 3.
			This proves the implication $\eqref{pt3-v2306}\Rightarrow\eqref{pt1-v2306}$
			in this case.
		  \item $\widetilde{A}$ in singular: First we show that the last column of
		  $\widetilde A$ is in the span of others. We separate two cases according to $k$.
		  \begin{itemize}
			\item $k=3$:	
			 Since $\widetilde{A}=
					\left(\begin{array}{cc}
						\beta_0 & \beta_3\\
						\beta_3 & \beta_{6}
					\end{array}\right)$
			and $\beta_0>0$, it follows that the second (also the last) column of $\widetilde{A}$ is
			a multiple of the first (also it the span of the others).
			\item $k>3$:
				Since $\overline A\succ 0$, the last column of 
				 $\widetilde{A}=
					\left(\begin{array}{cc}
						\overline{A} & r^T\\
						r & \beta_{2k}
					\end{array}\right)$ is in the span of the others, where 
			$r=\left(\begin{array}{cccc}
				\beta_k & \beta_{k+2}&\cdots& \beta_{2k-1}
				\end{array}\right)$.
			\end{itemize}
			Since $A_{\overline \beta}\succ 0$, it follows that 
			$A_{\overline{\overline \beta}}\succ 0$ and 
			$\widetilde A/A_{\overline{\overline \beta}}=0.$
			By the form of $y_\pm$ given in Claim 2, we have that $y_-=y_+$.
			By Claim 2, 
			$\widetilde A(y_+)=
			\left(\begin{array}{cc}
						\overline A(y_+) & r_1^T\\
						r_1 & \beta_{2k}
					\end{array}\right)\succeq 0,
			$
			where
			$r_1=\left(\begin{array}{ccc}
				\beta_k & \cdots& \beta_{2k-1}
				\end{array}\right)$,
			and $\Rank \widetilde A(y_+)=\Rank A_{\overline \beta}=k-1$.
			By Lemma \ref{extension-principle}, the last column of 
			$\widetilde A(y_+)$ is in the span of the others and hence 
			$\overline A(y_+)\succ 0$.
			Since by assumption also the inequality \eqref{cond-to-solve-v2306} is strict,
			$A_{\overline{\beta}(y_+)}\succ 0$ by Claim 3. 
			Hence,  \eqref{cond1-proof-2306} of Claim 1 holds for $y_0=y_+$, which
			proves the implication $\eqref{pt3-v2306}\Rightarrow\eqref{pt1-v2306}$
			in this case.	
	\end{itemize}
  \item Case \eqref{subcase2-2606}: 
	$\beta(x,sA_{\overline{\beta}}^{+}s^T )$ is ppsd. 
	Since 
	$$A_{\overline\beta(sA_{\overline{\beta}}^{+}s^T)}=
	\left(\begin{array}{cc}
		A_{\widetilde\beta(sA_{\overline{\beta}}^{+}s^T)} & r_2^T\\
		r_2 & \beta_{\beta_{2k}}
	\end{array}\right),$$
	where 
	$r_2=
		\left(\begin{array}{ccc}
			\beta_{k+1} & \cdots & \beta_{2k-1}
		\end{array}\right)$, is singular, the assumption 
		$uA_{\widetilde \beta}^+u^T<sA_{\overline\beta}^+s^T$ and Claim 3
		imply that $A_{\widetilde\beta(sA_{\overline{\beta}}^{+}s^T)}\succ 0$,
		hence
		$\Rank A_{\widetilde\beta(sA_{\overline{\beta}}^{+}s^T)}=
		\Rank A_{\overline\beta(sA_{\overline{\beta}}^{+}s^T)}=\Rank A_{\overline\beta}.$ 
		Hence, \eqref{cond2-proof-2306} of Claim 1 for 
		$y_0=sA_{\overline{\beta}}^{+}s^T$ holds, which
		proves the implication $\eqref{pt3-v2306}\Rightarrow\eqref{pt1-v2306}$
		in this case.		
  \item Case \eqref{case3-2306}: 
	By assumption $\Rank A_{\widetilde \beta}=\Rank A_{\overline{\beta}}$,
	it follows that the last column of $A_{\overline{\beta}}$
	is in the span of the others.
	There exists $x_0\in \RR$ such that $A_{\beta(x_0,y_+)}$ 
	is psd and by Lemma \ref{extension-principle}, 
	the last column of $A_{\overline \beta(y_+)}$ is in the span of the others and hence
	$\Rank A_{\widetilde \beta(y_+)}=\Rank A_{\overline \beta(y_+)}$.
	Since $A_{\overline \beta(y_+)}$ is singular, using Corollary \ref{rank-theorem-3} with $\beta$ equal to $\beta(x_0,y_+)$, 
	we get $\Rank A_{\overline \beta(y_+)}=\Rank A_{\overline \beta}$, which in particular implies that $y_+=sA_{\overline \beta}^+s^T$.
	Hence, $\Rank A_{\widetilde \beta(y_+)}=\Rank A_{\overline \beta(y_+)}=\Rank A_{\overline \beta}$, which
	is \eqref{cond2-proof-2306} of Claim 1.
	This proves the implication $\eqref{pt3-v2306}\Rightarrow\eqref{pt1-v2306}$
	in this case.
\end{itemize}

It remains to prove the implication $\eqref{pt1-v2306}\Rightarrow\eqref{pt2-v2306}$. 
By Theorem \ref{trunc-Hamb-without-1}, if $\beta(x,y_0)$ has a representing measure, then
there is a $(\Rank \overline \beta(y_0))$ or  $(\Rank \overline \beta(y_0)+1)$-atomic representing measure.
By Corollary \ref{rank-theorem-3}, $\Rank \overline \beta(y_0)=\Rank A_{\overline \beta(y_0)}=\Rank A_{\overline\beta}$ if 
$A_{\overline \beta(y_0)}$ is singular
and $\Rank \overline \beta(y_0)=\Rank A_{\overline \beta}+1=\Rank \overline\beta+1$ otherwise.

For the moreover part, note from the previous paragraph that $(\Rank \overline \beta)$-atomic measure exists if and only if 
$A_{\overline \beta(y_0)}=\Rank A_{\overline \beta}$ for some $y_0$ such the $\beta(x,y_0)$ admits a measure. 
The only $y_0\in \RR$ satisfying $\Rank A_{\overline \beta(y_0)}=\Rank A_{\overline \beta}$ is $sA_{\overline{\beta}}^{+}s^T$
and hence a $(\Rank \overline \beta)$-atomic measure exists if and only if $\beta(x,sA_{\overline{\beta}}^{+}s^T)$ admits a measure. 
From the proof of the implication 
$\eqref{pt3-v2306}\Rightarrow\eqref{pt1-v2306}$ we see that this is true 
in the cases \eqref{subcase2-2606} and \eqref{case3-2306}.
Finally, if  \eqref{subcase1-2606} holds, then we see that:
\begin{itemize} 
	\item If $\widetilde A\succ 0$, then we must have 
			$y_-\leq sA_{\overline{\beta}}^{+}s^T$
		and 
			$uA_{\widetilde \beta}^+u^T<sA_{\overline\beta}^+s^T$ (see the
			proof of  \eqref{subcase2-2606}), which means that \eqref{subcase2-2606} holds.
	\item If $\widetilde A$ is singular,
		then $sA_{\overline{\beta}}^{+}s^T<y_-=y_+$ and 
		$\beta(x,sA_{\overline{\beta}}^{+}s^T)$ does not admit a 
		$(\Rank \overline \beta)$-atomic measure.
\end{itemize}
This establishes the proof of the moreover part.
\end{proof}

\begin{remark}
	For $k=2$, the THMP with gaps $(\beta_1,\beta_2)$ coincides with the THMP with gaps 
	$(\beta_{2k-2},\beta_{2k-1})$
	and hence the case $k=2$ is already covered by  Theorem \ref{trunc-Hamb-without-3n-2-and-1} .
\end{remark}

The following corollary is a consequence of Theorem \ref{trunc-Hamb-without-1-2}  and solves the bivariate TMP for the curve $y^3=x^4$ where also $\beta_{\frac{5}{3},0}$ is given. Here $\beta_{\frac{5}{3},0}$ stands for the integral
of $x^{\frac{5}{3}}$ w.r.t.\ $\mu$, i.e., $\int_K x^{\frac{5}{3}}d\mu$. 

\begin{corollary}\label{Y3=X4-gen} 
	Let $\displaystyle\beta=(\beta_{i,j})_{i,j\in \ZZ^2_+,i+j\leq 2k}$ $\beta$ be a 2-dimensional real 	
	multisequence of degree $2k$ and let $\beta_{\frac{5}{3},0}$ be also given. 
	Suppose $M(k)$ is positive semidefinite and recursively generated.
	Let 
	$$u^{(1)}=(\beta_{0,1},\beta_{\frac{5}{3},0},\beta_{2,0},\beta_{1,2}),\quad 
	u^{(i)}=(\beta_{0,i},\beta_{3,i-2},\beta_{2,i-1},\beta_{1,i})\quad \text{for}\quad i=2,\ldots,2k-1,$$ 
	$$\widetilde{\beta}:=(u^{(1)},\ldots,u^{(2k-2)},\beta_{0,2k-1},\beta_{3,2k-3},\beta_{2,2k-2}),\quad
		\overline\beta :=(\widetilde{\beta},\beta_{1,2k-1},\beta_{0,2k}),$$
	$$
			\breve\beta:=(\widehat{\beta},\beta_{3,2k-3},\beta_{2,2k-2})
	\quad\text{and}\quad
				\overline{\overline{\beta}}:=(\breve\beta,\beta_{3,2k-1},\beta_{0,2k})
	$$
	be subsequences of $\beta$,
				$$v:=\left(\begin{array}{cccccccc}
					\beta_{1,0} & u^{(1)} & \cdots & u^{(k-2)} & \beta_{0,k-1} & 
					\beta_{3,k-3}&\beta_{2,k-2} & \beta_{1,k-1}\end{array}\right),\quad
				u:=\left(\begin{array}{cc}v & \beta_{0,k}\end{array}\right),$$
				$$s:=\left(\begin{array}{cc} u & \beta_{3,k-2}\end{array}\right),\quad
				w:=\left(\begin{array}{ccccccccc}
					\beta_{\frac{5}{3},0} & \beta_{2,0} &\beta_{1,1} & u^{(2)} & \cdots &
							u^{(k-1)} & \beta_{0,k} & \beta_{3,k-2} & \beta_{2,k-1}\end{array}\right)$$
	vectors and
		$$\overline{A}:=\left(\begin{array}{cc} \beta_0 & v \\ v^T & A_{\breve \beta}\end{array}\right)
			\quad\text{and}\quad
			\widetilde{A}:=
			\left(\begin{array}{cc} \beta_0 & u \\ u^T & A_{\overline{\overline{ \beta}}}\end{array}		
			\right)$$
	matrices.
	Then $\beta$ has a representing measure supported on $y^3=x^4$ if and only if 
	\begin{equation}\label{cond-to-solve-v270720}
		 	sA_{\overline{\beta}}^{+}s^T 
			\leq u A_{\overline{\overline\beta}}^+w^T+\sqrt{(A_{\overline{\beta}}/A_{\overline{\overline{\beta}}}) (\widetilde A/A_{\overline{\overline \beta}})}
	\end{equation}
	one of the following statements hold: 
	\begin{enumerate}
	\item One of the following holds: 
			\begin{itemize}
				\item If $k\geq 4$, then $Y^3=X^4$ is a column relation of $M(k)$. 
				\item If $k=3$, then the equalities 
					$\beta_{0,3}=\beta_{4,0}$, $\beta_{1,3}=\beta_{5,0}$, $\beta_{2,3}=\beta_{6,0},$ 
					$\beta_{0,4}=\beta_{4,1}$, $\beta_{0,5}=\beta_{4,2}$.
				\item If $k=2$, then the equality 
					$\beta_{0,3}=\beta_{4,0}$ holds.
				\item $k=1$.
			\end{itemize}
	\item One of the following holds: 
		\begin{enumerate}
			\item\label{cond2-09-07-20} $A_{\overline{\beta}}\succ 0$, $\overline A\succ 0$ and the inequality in 
				\eqref{cond-to-solve-v270720} is strict.
			\item\label{subcase2-170720} $A_{\overline{\beta}}\succ 0$ and the following inequalities holds:
				$$uA_{\widetilde{\beta}}^{+}u^T < sA_{\overline{\beta}}^{+}s^T
					\quad\text{and}\quad 
					u A_{\overline{\overline\beta}}^+w^T-
					\sqrt{(A_{\overline{\beta}}/A_{\overline{\overline{\beta}}}) 
					(\widetilde A/A_{\overline{\overline \beta}})}
					\leq sA_{\overline{\beta}}^{+}s^T.$$
			\item\label{pt3-170720} $A_{\overline\beta}\succeq 0$ and 
				$\Rank A_{\widetilde \beta}=\Rank A_{\overline \beta}=
					\Rank \left(\begin{array}{cc} s^T & A_{\overline \beta}\end{array}\right)$.
			\end{enumerate}
	\end{enumerate}
Moreover, if the representing measure exists, then there exists a $(\Rank \overline\beta)$-atomic measure
if and only if \eqref{subcase2-170720} or \eqref{pt3-170720} holds. 
	Otherwise there is a $(\Rank \overline\beta+1)$-atomic measure
\end{corollary}

\begin{proof}
	For $\{0,3,4,6,\ldots,8k\}$ we define the numbers $\widetilde \beta_m$ by the following rule
		\begin{equation*}
			\widetilde \beta_m:=
			\left\{\begin{array}{rl}
				\beta_{0,\frac{m}{4}},& \text{if }m\Mod{4}=0,\\
				\beta_{3,\lfloor\frac{m}{4}\rfloor-2},& \text{if }m\Mod{4}=1,\\
				\beta_{2,\lfloor\frac{m}{4}\rfloor-1},& \text{if }m\Mod{4}=2,\\
				\beta_{1,\lfloor\frac{m}{4}\rfloor},& \text{if }m\Mod{4}=3.
			\end{array}\right.
		\end{equation*}
	
	\noindent \textbf{Claim 1.} Every number $\widetilde \beta_m$ is well-defined.\\

	We have to prove that $i+j\leq 2k$, where $i,j$ are indices of $\beta_{i,j}$ used in the definition of 
	$\widetilde \beta_m$.
	We separate four cases according to $m$:
	\begin{itemize}
		\item $m\Mod{4}=0$: $\frac{m}{4}\leq 2k$.
		\item $m\Mod{4}=1$: $\lfloor\frac{m}{4}\rfloor-2+3\leq (2k-1)+1=2k.$
		\item $m\Mod{4}=2$: $\lfloor\frac{m}{4}\rfloor-1+2\leq (2k-1)+1=2k.$
		\item $m\Mod{4}=3$: $\lfloor\frac{m}{4}\rfloor+1\leq (2k-1)+1=2k.$\\
	\end{itemize}
	
	We also define
		$\widetilde \beta_5:=\beta_{\frac{5}{3},0}.$\\

	\noindent \textbf{Claim 2.} 
		Let $t\in \NN$. 
		The atoms $(x_1^3,x_1^4),\ldots (x_t^3,x_t^4)$ with densities $\lambda_1,\ldots,\lambda_t$
		are the $(y^3-x^4)$-representing measure for $\beta$ with $\beta_{\frac{5}{3},0}$ known
		if and only if
		the atoms $x_1,\ldots,x_t$ with densities $\lambda_1,\ldots,\lambda_t$
		are the $\RR$-representing measure for 
		 $\widetilde \beta(x,y)=(\widetilde \beta_0,x,y,\widetilde \beta_3,\ldots,\widetilde\beta_{2k})$.\\

	The if part follows from the following calculation:
		$$\widetilde \beta_{m}
			=\left\{\begin{array}{rl}
				\beta_{0,\frac{m}{4}},& \text{if }m\Mod{4}=0,\\
				\beta_{3,\lfloor\frac{m}{4}\rfloor-2},& \text{if }m\Mod{4}=1,\\
				\beta_{2,\lfloor\frac{m}{4}\rfloor-1},& \text{if }m\Mod{4}=2,\\
				\beta_{1,\lfloor\frac{m}{4}\rfloor},& \text{if }m\Mod{4}=3,\\
				\end{array}\right.
			=\left\{\begin{array}{rl}
				\sum_{\ell=1}^t \lambda_\ell (x_{\ell}^4)^{\frac{m}{4}},& \text{if }m\Mod{4}=0,\\
				\sum_{\ell=1}^t \lambda_\ell (x_{\ell}^3)^{3}(x_{\ell}^4)^{\lfloor\frac{m}{4}\rfloor-2},& 
					\text{if }m\Mod{4}=1,\\
				\sum_{\ell=1}^t \lambda_\ell (x_{\ell}^3)^{2}(x_{\ell}^4)^{\lfloor\frac{m}{4}\rfloor-1},& 
					\text{if }m\Mod{4}=2,\\
				\sum_{\ell=1}^t \lambda_\ell x_{\ell}^3(x_{\ell}^4)^{\lfloor\frac{m}{4}\rfloor},& 
					\text{if }m\Mod{4}=3,\\
				\end{array}\right.
			=\sum_{\ell=1}^t \lambda_\ell x_\ell^{m},$$
	where $m=0,3,4,6,\ldots,8k$ and 
		$$\widetilde \beta_5=\beta_{\frac{5}{3},0}=
				\sum_{\ell=1}^t \lambda_\ell (x_{\ell}^3)^{\frac{5}{3}}=
				\sum_{\ell=1}^t \lambda_\ell x_\ell^{5}.$$

	The only if part follows from the following calculation:
	\begin{align*}
	  \beta_{i,j}
		&= \beta_{i-4,j+3}=\cdots=\beta_{i\Mod{4},j+3\lfloor\frac{i}{4}\rfloor}
		  =\widetilde \beta_{3(i\Mod{4})+4(j+3\lfloor\frac{i}{4}\rfloor)}\\
		&= \sum_{\ell=1}^t \lambda_\ell x_\ell^{3(i\Mod{4})+4(j+3\lfloor\frac{i}{4}\rfloor)}
		  =\sum_{\ell=1}^t \lambda_\ell x_\ell^{3(i\Mod{4}+4\lfloor\frac{i}{4}\rfloor)} x_{\ell}^{4j}
		  =\sum_{\ell=1}^t \lambda_\ell (x_\ell^3)^{i}(x_{\ell}^4)^{j},
	\end{align*}
	where the first three equalities in the first line follow by $M(k)$ being rg and
		$$\beta_{\frac{5}{3},0}=\widetilde \beta_5
			=\sum_{\ell=1}^t \lambda_\ell x_\ell^{5}
			=\sum_{\ell=1}^t \lambda_\ell (x_{\ell}^3)^{\frac{5}{3}}.$$

	Using Claim 2 and a theorem of Bayer and Teichmann \cite{BT06}, implying that if a finite sequence 
	has a $K$-representing measure, then it has a finitely atomic $K$-representing measure, the statement of the 	
	Corollary follows by Theorem \ref{trunc-Hamb-without-1-2}.
\end{proof}

\end{document}